\documentclass[11pt,a4paper]{amsart}

\usepackage{amsfonts,amsmath,amssymb,amsthm}

\usepackage[utf8]{inputenc}
\usepackage[T1]{fontenc}

\usepackage{hyperref}
\usepackage{stmaryrd}
\usepackage{enumitem}

\newcommand{\Cc}{\mathbb{C}} 
\newcommand{\Pp}{\mathbb{P}}
\newcommand{\Rr}{\mathbb{R}}
\newcommand{\Nn}{\mathbb{N}}
\newcommand{\Zz}{\mathbb{Z}}
\newcommand{\Qq}{\mathbb{Q}}
\newcommand{\Ff}{\mathbb{F}}
\newcommand{\Aa}{\mathbb{A}}

\renewcommand {\leq}{\leqslant}
\renewcommand {\geq}{\geqslant}

\renewcommand {\epsilon}{\varepsilon}

{\theoremstyle{plain}
\newtheorem{theorem}{Theorem}[section]    % theorem with number
\newtheorem{lemma}[theorem]{Lemma}       % lemma with number
\newtheorem{proposition}[theorem]{Proposition}      % lemma with number
\newtheorem{corollary}[theorem]{Corollary}      % lemma with number
\newtheorem{conjecture}[theorem]{Conjecture} 

\newtheorem*{theorem*}{Theorem}
\newtheorem*{hypothesis*}{Schinzel Hypothesis}
}
{\theoremstyle{remark}
\newtheorem{definition}[theorem]{Definition}      % lemma with number
\newtheorem{question}[theorem]{Question} 
\newtheorem{notation}[theorem]{Notation} 
\newtheorem{remark}[theorem]{Remark}   % theorem without number
\newtheorem{example}[theorem]{Example}
}

\usepackage[a4paper]{geometry}
\title[Hurwitz spaces and Inverse Galois Theory]{Hurwitz spaces and Inverse Galois Theory}

\author{Pierre D\`ebes}

\email{pierre.debes@univ-lille.fr}

\address{Universit\'e de Lille, CNRS, Laboratoire Paul Painlev\'e, 59000 Lille, France}



\thanks{\emph{Acknowledgments}. 
The author would like to thank Angelot Behajaina, Benjamin Collas, Daniel Loughran, Shinichi Mochizuki, B\'eranger Seguin and the anonymous referee for valuable comments on a first version of the manuscript and numerous interesting exchanges. The author also acknowledges the support of the CDP C2EMPI project under the Initiative of Excellence of the University of Lille. The author is also grateful to Tel Aviv University, the Technion Institute of Technology in Haifa, the Research Institute of Mathematical Sciences in Kyoto and Osaka University for their hospitality. 
}

\date{\today}

\begin{document}

\begin{abstract}
Hurwitz spaces which parametrize 
branched covers of the line play a prominent role in inverse Galois theory. 
This paper surveys fifty years of works in this direction
with emphasis on recent advances.
Based on the Riemann-Hurwitz theory of covers, the geometric and arithmetic setup is first reviewed, 
followed by the semi-modern developments of the 1990--2010 period: 
large fields, 
compactification,
descent theory, modular towers. The second half  of the paper highlights more recent achievements that have reshaped the arithmetic of Hurwitz spaces, notably via the systematic study of 
the ring of components. These include the construction of components defined over $\Qq$,
and the Ellenberg-Venkatesh-Westerland approach to rational points over finite fields, applied to the Cohen-Lenstra heuristics and the Malle conjecture over function fields $\Ff_q(T)$.
\end{abstract}

\dedicatory{Dedicated to Hiroaki Nakamura on the occasion of his 60th birthday}

\maketitle

\begingroup
\renewcommand\thefootnote{}
\footnotetext{This paper will appear in ``Low Dimensional Topology, Number Theory and
Arithmetic Galois Theory'', edited by Pierre D\`ebes, Masanori Morishita, Takao Satoh,
Akio Tamagawa, to be published by World Scientific.}
\endgroup

%%%%%%%%%%%%%%%%%%%%%%%%%%%%%%%%%%%%%%%%%%%%%%%%%%%%%%%%%%%%%%%%%

This paper surveys the Hurwitz space theory as it fits within inverse Galois theory. The history goes back to Fried's 1977 paper \cite{fried77}: following Riemann and Hurwitz, he offers a general construction of \emph{Hurwitz moduli spaces} of 
branched covers of the projective line
and explains how they can be used to realize finite groups as Galois groups. This program will be successfully pursued by Matzat and his school for many finite groups. Section \ref{sec:RH-construction} reviews this first part with emphasis on the fundamentals, 
which remain essential in the subsequent developments. Fundamentals include \emph{patching}, a parallel $p$-adic character of the theory. Section \ref{sec:further-developments} is devoted to the \emph{semi-modern} developments of the theory: Hurwitz spaces as schemes, compactification, rational points over large fields, descent theory, modular towers. They happened between 1990 and 2010 and consolidated the theory. The contents of the first two sections is much documented in the literature. Here
I highlight the main ideas, tools, results and their articulation, and provide references\footnote{\hbox{Including to previous surveys: \cite{DeZakopane} on basics, \cite{De-Berkeley} on descent theory,  \cite{DeMT} on modular towers.}}. The last two sections emphasize the more recent achievements and are more detailed. 
The search of components of Hurwitz spaces defined over $\Qq$ has been a major motivation since the 2000s. Section \ref{sec:components} discusses the latest results 
along with some applications. Our presentation is mainly motivated by arithmetic questions and rationality issues over \emph{number fields}. In the last decade though, significant results on rational points of Hurwitz spaces over \emph{finite fields}, with striking arithmetic applications, have been obtained. 
\hbox{Section \ref{sec:finite-fields} is devoted to this part.}
\vskip 1mm

The paper is based on several lectures I have given in the recent years: a mini-course in Israel in 2023, a series of talks in Japan between 2023 and 2025 within the AHGT project, and a talk at the Low Dimension Topology and Number Theory XVI conference held in Osaka in 2025 in honor of Professor Hiroaki Nakamura's 60th birthday. 

\vskip 1,5mm
\noindent
{\emph{Hiroaki Nakamura and I first met in 1995, twice that year -- in Heidelberg and in Luminy. Our mentors and research directions made us something like cousins within the family of geo\-me\-tric Galois theory. Over the years, we have shared many ideas in numerous conversations. This paper is an opportunity to express my gratitude for his continued interest, \hbox{e.g.} in Hurwitz spaces, and my admiration for the depth and breadth of his contributions.}}
\vskip -1mm

\vskip -30mm
\tableofcontents

\vskip -5mm

\noindent
\textit{Galois preamble}. {{{Inverse Galois theory}}} is the main theme in the background, with at its very core the {Inverse Galois Problem}, which we recall as a conjectural statement.

\vskip 2mm

\noindent
{\bf Inverse Galois Problem}. \emph{
Every finite group is the Galois group of some Galois field extension of the field $\Qq$ of rational numbers.}
\vskip 2mm

This problem is still open and notoriously difficult. There are very few groups for which it is ``easy'', apart from small cyclic groups and symmetric groups $S_n$.\footnote{Even for arbitrary cyclic groups, the common proof  uses Dirichlet's theorem on primes in arithmetic progressions. As to the symmetric group $S_n$, an elementary argument due to van der Waerden shows that for $n\geq 3$, it is the Galois group of the polynomial $X^n-X-1$.} Most groups require advanced strategies.
There is a natural one for solvable groups: one basically tries to pile up abelian extensions in a consistent manner. The strategy is simple but working it out was a real \emph{tour de force}, a big 
achievement due to Shafarevich (1954).
\vskip 1mm

Another idea is to first work over a bigger and easier field $k$ (instead of $\Qq$) and then see the remaining problem as a \emph{descent problem}. A classical strategy, going back to Riemann and Hurwitz, follows this line with $k=\Cc(T)$.
The beginning stage, which we review in Section \ref{ssec:RH-construction}, is this striking result:
\vskip 2mm

\noindent
{\bf Inverse Galois Problem over $\Cc(T)$.} \label{cor:IGP-C(T)} {\it Every finite group is the Galois group of some and in fact of many Galois extensions $N/\Cc(T)$}. 
\vskip 2mm

A second stage addresses the question of whether among these ``many'' extensions $N/\Cc (T)$, at least one can be descended to a Galois extension $N_0/\Qq(T)$ with the same group $G$. It is this part that involves \emph{Hurwitz spaces}. The ultimate goal is the
\vskip 2mm

\noindent
{\bf Regular Inverse Galois Problem over $\Qq$}. \emph{
Every finite group is the Galois group of some field extension $N_0/\Qq(T)$ that in addition is regular over $\Qq$, \hbox{i.e.} $N_0\cap \overline \Qq=\Qq$.}
\vskip 2mm

The final stage, from $\Qq(T)$ to $\Qq$, is classically handled by the celebrated Hilbert irreducibility theorem: the Regular Inverse Galois Problem implies the Inverse Galois Problem. 
This stage is perfectly understood, even 
at a quantitative level;
see \cite{IsrJ2} and \cite{motte23} which give lower bounds in $c y^{\alpha}$ for the number of Galois extensions of a number field $k$ with a prescribed group $G$ and a discriminant $\leq y$ obtained by specializing a $k$-regular Galois extension $E/k(T)$ of group $G$ (for some constants $c, \alpha>0$ depending on $E/k(T)$). 
\vskip 1mm

We emphasize first and second stages: --over $\Cc(T)$-- and --from $\Cc(T)$ to $\Qq(T)$--. They are presented in Section \ref{sec:RH-construction}; we refer to them as the \emph{Riemann-Hurwitz strategy}.

\section{Hurwitz spaces} \label{sec:RH-construction}

Section \ref{ssec:RH-construction} recalls the main steps of the first stage of the Riemann-Hurwitz strategy, leading to the Inverse Galois Problem over $\Cc(T)$, and to the definition of Hurwitz spaces as complex analytic spaces. Section \ref{sec:Hurwitz-spaces} focuses on the algebraic structure and the moduli properties of Hurwitz spaces. Section \ref{sec:descent} explains the second stage of the Riemann-Hurwitz strategy:  how to use Hurwitz spaces for the descent from $\Cc(T)$ to $\Qq(T)$. Section \ref{sec:patching} reviews {\it patching}, another central tool in inverse Galois theory,
  which can be seen as a $p$-adic avatar of the Riemann-Hurwitz construction. Section \ref{ssec:nonGalois} is a final note on the use of Hurwitz spaces for non-Galois covers, which was the context of Hurwitz' original motivation.

\subsection{The Riemann-Hurwitz construction}  \label{ssec:RH-construction}
See \cite{De_St-Etienne2} for more details, and \cite{De_St-Etienne3} for a full presentation. Let $G$ be a finite group.
\vskip 1mm

\subsubsection{Step 1 \hbox{\rm [Initial data]}} Pick an integer $r\geq 2$, a generating system $\{g_1,\ldots, g_{r}\}\subset G\setminus\{1\}$ of the group $G$ such that $g_1\cdots g_r=1$, and $r$ distincts points $t_1,\ldots,t_r\in \Aa^1(\Cc)$. 
\vskip 0,5mm

\noindent
(This is possible provided that $r$ is bigger than the rank of $G$). 

\subsubsection{Step 2 \hbox{\rm [Topological stage]}} Set  $\underline t=\{t_1,\ldots,t_r\}$ and denote the fundamental group of $\Pp^1(\Cc)\setminus \underline t$ based at the point at infinity  by $\pi_1(\Pp^1(\Cc)\setminus \underline t,\infty)$. Fix a ``standard loop'' $\gamma_i$ based at $\infty$ and revolving once around $t_i$ in the anticlockwise direction, $i=1,\ldots,r$. Consider the group homomorphism $\varphi: \pi_1(\Pp^1(\Cc)\setminus \underline t,\infty) \rightarrow G$ that maps the homotopy class $[\gamma_i]$ to the element $g_i\in G$, $i=1,\ldots,r$. 
\vskip 0,5mm

\noindent
(By ``standard loop'' we mean that the tuple $\underline \gamma=(\gamma_1, \ldots, \gamma_r)$ should satisfy some additional technical conditions making $\underline \gamma$ a {\emph{topological bouquet} for $\Pp^1(\Cc)\setminus \underline t$ at $\infty$ (\hbox{e.g.} \cite[\S 1.1]{DeEm04}). A consequence is that   $[\gamma_1], \ldots, [\gamma_r]$ generate the fundamental group $\pi_1(\Pp^1(\Cc)\setminus \underline t,\infty)$ with the  one condition $[\gamma_1] \cdots[\gamma_r] = 1$. This guarantees that the homomorphism $\varphi$ is well-defined).
\vskip 1mm

From fundamental group theory, the group homomorphism $\varphi$, composed with the regular representation $G \rightarrow S_{|G|}$, gives rise to a topological (unramified) cover $f:X \rightarrow \Pp^1(\Cc)\setminus \underline t$, which is Galois of group $G$. 
The upper space $X$ is connected as a consequence of the homomorphism $\varphi$ being surjective (onto $G$).

\subsubsection{Step 3 \hbox{\rm [Completion]}} \label{ssec:completion}The topological space $X$ inherits the analytic structure of $\Pp^1(\Cc)\setminus \underline t$ to become 
a Riemann surface, and  $f:X \rightarrow \Pp^1(\Cc)\setminus \underline t$ is then a finite analytic cover. Using classical arguments from Riemann surface theory, one can complete the cover $f$ above the missing points $t_1,\ldots, t_r$. More precisely, one can define an analytic map of compact Riemann surfaces $\overline f:\overline X \rightarrow \Pp^1(\Cc)$ such that its restriction to $\Pp^1(\Cc)\setminus \underline t$ is isomorphic to $f$ (as topological and analytic covers).

\subsubsection{Step 4 \hbox{\rm [Algebraization]}} \label{ssec:algebraization} The deeper part of the Riemann-Hurwitz construction shows that the meromorphic function field extension $M(\overline X)/M(\Pp^1(\Cc))$ associated to the map $\overline f$ is a finite Galois extension of group $G$, and that the field $M(\Pp^1(\Cc))$ is isomorphic to $\Cc(T)$.

\subsubsection{{\rm [Outcome of the Riemann-Hurwitz construction]}} 
 {\it The promised Galois extensions $N/\Cc(T)$ of group $G$ are the extensions $M(\overline X)/M(\Pp^1(\Cc))$, where $\underline t=\{t_1,\ldots,t_r\}$ may vary over the set ${\rm Conf}_r(\Cc)$ of all subsets of $r$ distinct elements of $\Aa^1(\Cc)$.  }
\vskip 2mm

This solves in particular the Inverse Galois Problem over $\Cc(T)$.
\vskip 1mm

\begin{remark}[Riemann Existence Theorem] The arguments mentioned above give more than the mere construction of Galois extensions $N/\Cc(T)$. 
They 
provide the foundation of several equivalent categorical descriptions of branched covers of the line. These are specifically stated in \S \ref{ssec:cat-G-covers} and commented in Remark \ref{rem:RET-cont}, and commonly referred to as  the \emph{Riemann Existence Theorem}. 
\end{remark}

\subsubsection{Configuration space and Hurwitz set} \label{ssec:conf-space-H-space} \label{ssec:hurwitz set}
The set 
\vskip 1mm

\centerline{${\rm Conf}_r(\Cc)=\{ \underline t =\{t_1,\ldots,t_r\} \hskip 2pt | \hskip 2pt t_i \in \Aa^1(\Cc)\ \hbox{\rm  and }   t_i\not= t_j \ \hbox{\rm  if } i\not= j\ (i,j=1,\ldots,r)\}$} 
\vskip 1mm

\noindent
is called the \emph{configuration space for branch point data}. The idea behind 
Hurwitz spaces is to perform the Riemann-Hurwitz construction in families, above ${\rm Conf}_r(\Cc)$. As a set of complex points, the \emph{Hurwitz space}
is defined by 
\vskip 1mm

\centerline{${\rm Hur}_{r,G}(\Cc) = \left\{ (\underline t, \overline \varphi) \hskip 1pt | \hskip 2pt \underline t\in {\rm Conf}_r(\Cc)\hbox{ and } \left\{\begin{matrix}\varphi \in {\rm Hom}(\pi_1(\Pp^1(\Cc)\setminus \underline t,\infty),G)\\ \varphi([\gamma_i]) \not=1,\hskip 1mm  i=1,\ldots,r \hfill \end{matrix}\right.\hskip 2mm \right\}$.}
\vskip 1mm

\noindent
By $\overline \varphi$, we mean the \emph{isomorphism class} of the group homomorphism $\varphi$; 
the isomorphisms are specified in \S \ref{ssec:isomorphisms}.  For simplicity, we often omit the bar, \hbox{i.e.} write $\varphi$ for $\overline \varphi$ in the notation. The second condition on $\varphi$ guarantees that the points $t_1,\ldots,t_r$ are indeed branched.
\vskip 1mm 
The definition of ${\rm Hur}_{r,G}(\Cc)$ reflects the conclusion of the Riemann-Hurwitz construction that complex Galois covers with given group and fixed branch point number can be classified by their branch point set $\underline t$, a continuous data, conjoined with some discrete data, the homomorphism $\varphi$. 
The Hurwitz set ${\rm Hur}_{r,G}(\Cc)$ inherits the analytic structure of ${\rm Conf}_{r}(\Cc)$ to become a smooth analytic space for which the natural map $\Phi: {\rm Hur}_{r,G}(\Cc) \rightarrow {\rm Conf}_{r}(\Cc)$ is a finite analytic cover.

\subsection{Hurwitz spaces as moduli spaces} \label{sec:Hurwitz-spaces}
Making the Hurwitz space a \emph{moduli space} in the full algebraic sense requests thinking in terms of categories. So does the descent part of the \hbox{Riemann-Hurwitz} strategy --from $\Cc(T)$ to $\Qq(T)$-- (explained in \S \ref{sec:descent}).

\subsubsection{The category of G-covers}  \label{ssec:cat-G-covers} \hskip1mm
\vskip 1mm

\noindent
{\ref{ssec:cat-G-covers}.1. {\it The objects of the category.} The Riemann-Hurwitz construction shows that there are one-one correspondences between 

(a) points $(\underline t, \varphi)\in {\rm Hur}_{r,G}(\Cc)$ (as defined above in \S \ref{ssec:conf-space-H-space})

\noindent
and isomorphism classes

(b) of topological (unramified) covers $f:X\rightarrow \Pp^1(\Cc)\setminus \underline t$ given with a homomorphism $G \rightarrow {\rm Aut}(f)$ such that the action of  $G$ on the fibers is free and transitive, or, equivalently,

 (c) of finite algebraic morphisms $f:X\rightarrow \Pp^1_\Cc$ ramified exactly above every point in $\underline t$, with $X$ smooth, given with a homomorphism $G \rightarrow {\rm Aut}(f)$ such that the action of   $G$ on the unramified fibers is free and transitive, or, equivalently, 

(d) of Galois extensions $F/\Cc(T)$ ramified exactly above every point in $\underline t$, given with an injective homomorphism 
{\hbox{${\rm Gal}(F/\Cc(T)) \hookrightarrow G$.}}
\vskip 1mm

\begin{remark}[Riemann existence theorem (continued)] \label{rem:RET-cont}
Here are some indications on the various correspondences. The Riemann-Hurwitz construction from \S \ref{ssec:RH-construction} yields correspondences \hbox{${\rm (a)} \hskip -1mm \Rightarrow \hskip-1mm {\rm (b)}$} and \hbox{${\rm (b)} \hskip-1mm \Rightarrow \hskip -1mm {\rm (d)}$}. The normalization process provides \hbox{${\rm (d)} \hskip -1mm \Rightarrow \hskip-1mm {\rm (c)}$}: namely, the normalization $\mathcal X$ of $\Pp^1_\Cc$ in the function field $M(\overline X)$ yields a finite algebraic morphism $\mathcal X\rightarrow \Pp^1_\Cc$ ramified above $\underline t$. The inverse correspondence 
\hbox{${\rm (c)} \hskip -1mm \Rightarrow \hskip-1mm {\rm (d)}$} is provided by the function field functor.
Restriction above $\Pp^1(\Cc)\setminus \underline t$ of a finite algebraic morphism $\mathcal X\rightarrow \Pp^1_\Cc$ ramified above $\underline t$ induces a topological cover of $\Pp^1(\Cc)\setminus \underline t$
 that is topologically  equivalent to the original cover $f$, thus giving \hbox{${\rm (c)} \hskip -1mm \Rightarrow \hskip-1mm {\rm (b)}$}. Finally \hbox{${\rm (b)} \hskip -1mm \Rightarrow \hskip-1mm {\rm (a)}$} goes like this.
Starting from a topological cover $f:X\rightarrow \Pp^1(\Cc)\setminus \underline t$ and a topological bouquet $\underline \gamma = (\gamma_1,\ldots,\gamma_r)$ of $\Pp^1(\Cc)\setminus \underline t$ based at $\infty$, consider the monodromy permutations of the fiber $f^{-1}(\infty)$ associated with the loops $\gamma_1,\ldots,\gamma_r$. These permutations are the \emph{branch cycles} of $f$ \hbox{w.r.t.}  $\underline \gamma$. Their tuple, which we denote by ${\rm BCD}_{\underline \gamma}(f)$ and call the \emph{branch cycle description} of $f$ \hbox{w.r.t.} $\underline \gamma$, gives rise to a homomorphism $\varphi \in {\rm Hom}(\pi_1(\Pp^1(\Cc)\setminus \underline t,\infty),G)$, and
so to a point $(\underline t, \varphi) \in {\rm Hur}_{r,G}(\Cc)$.
\end{remark}

\begin{definition} The equivalent objects in (a),(b),(c),(d) are called  \emph{G-covers of $\Pp^1_\Cc$} (or \emph{G-extensions of $\Cc(T)$} in the last case) with group $G$ and $r$ branch points. The subgroup $\varphi(\pi_1(\Pp^1(\Cc)\setminus \underline t,\infty)) \subset G$ is the monodromy group of the cover $f:X\rightarrow \Pp^1(\Cc)\setminus \underline t$ and the Galois group ${\rm Gal}(F/\Cc(T))$, up to isomorphism; it is called the \emph{group of the G-cover}. 
\end{definition}

\begin{remark}[Connectedness]  The general definition of G-covers does not assume that they are \emph{connected}, that is, that the upper 
space $X$ of the topological cover is connected, or, equivalently, irreducible, in algebraic terms. This is equivalent to the homomorphism $\varphi: \pi_1(\Pp^1(\Cc)\setminus \underline t,\infty)\rightarrow G$ being surjective (\hbox{i.e.} to the group of the cover being $G$), or, for G-extensions, to the monomorphism ${\rm Gal}(F/\Cc(T)) \hookrightarrow G$ being bijective. Recall however that the G-covers involved in the Riemann-Hurwitz construction \emph{are} connected. The relaxed definition allows more flexibility in Hurwitz space theory.
\end{remark}

\begin{remark}[G-covers over more general fields]Via the last two descriptions (c),(d), the definition of G-covers extends to the situation that $\Cc$ is replaced by any algebraically closed field $\overline k$. If $\overline k$ is of characteristic $0$ or prime to $|G|$, they correspond to points $(t,\varphi)$ in the set ${\rm Hur}_{r,G}(\overline k)$ defined similarly as ${\rm Hur}_{r,G}(\Cc)$ (\S \ref{ssec:hurwitz set}) with the difference that the topological fundamental group $\pi_1(\Pp^1(\Cc)\setminus \underline t,\infty),G)$ should be replaced by the \'etale fundamental group $\pi_1^{\hbox{\tiny\'et}}(\Pp^1(\Cc) \setminus \underline t,\infty)$\footnote{This \'etale fundamental group is the Galois group of the maximal algebraic extension $\Omega_{\underline t}/\overline k(T)$ unramified above ${\underline t}$, and $\infty$ here denotes a fixed embedding $\Omega_{\underline t} \hookrightarrow \overline{k(T)}$. By the Riemann existence theorem, this Galois group is the profinite completion of the topological fundamental group $\pi_1(\Pp^1(\Cc)\setminus \underline t,\infty),G)$.}. Over an arbitrary field $k$, G-covers can also be defined, by descriptions (c),(d), provided that, in (d), the extension $F/k(T)$ be further requested to be \emph{regular over $k$} (or \emph{$k$-regular}), \hbox{i.e.}, to satisfy $F\cap \overline k= k$. This amounts, in (c), to request that the morphism $f:X\rightarrow \Pp^1_k$ be geometrically irreducible.
\end{remark}

\noindent
{\ref{ssec:cat-G-covers}.2. {\it Isomorphisms between G-covers.} \label{ssec:isomorphisms} 
They are traditionally of two types:
\vskip 0,5mm

\noindent
\hskip 2mm 
- for {\it unmarked G-covers} : $\varphi_1, \varphi_2 \in {\rm Hom}(\pi_1(\Pp^1(\Cc)\setminus \underline t,\infty),G)$ are isomorphic if 
\vskip 0,5mm

\centerline{$\varphi_2 = g \varphi_1 g^{-1}$ for some $g\in G$.} 
\vskip 1mm 

\noindent
That is, the morphisms come from the inner automorphisms of $G$. The resulting category equivalently corresponds to that of algebraic morphisms $f:X\rightarrow \Pp^1_\Cc$ as in {\ref{ssec:cat-G-covers}.1(c), up to natural  isomorphisms, or, to that of G-Galois extensions $F/\Cc(T)$ up to $\Cc(T)$-conjugacy.

\vskip 1mm

\noindent
\hskip 2mm - for {\it marked G-covers}: $\varphi_1, \varphi_2 \in {\rm Hom}(\pi_1(\Pp^1(\Cc)\setminus \underline t,\infty),G)$ are isomorphic if they are equal. 
\vskip 0,5mm 

\noindent
That is, the only morphism is the identity. The resulting category equivalently corresponds to that of algebraic morphisms $f:X\rightarrow \Pp^1_\Cc$ as in {\ref{ssec:cat-G-covers}.1(c), given with a point in the fiber above $\infty$, up to natural isomorphism, or, equivalently, to G-Galois extensions $F/\Cc(T)$ given with a $\Cc(T)$-embedding $F \hookrightarrow \Cc((1/T))$.
\vskip 1mm

If $k$ is a subfield of $\Cc$, a $\Cc$-G-cover $f$ is said to be \emph{defined over $k$} if there is a $k$-G-cover $f_0$ such that $f$ and $f_0\otimes_k \Cc$ are isomorphic.

\begin{remark}[marked \hbox{vs.} unmarked] \label{rem:marked-unmarked1} The marked context is more adapted to the development of the theory. The unmarked context is closer to the Regular Inverse Galois Problem spirit: unmarked G-covers defined over a field $k$ are the $k$-regular extensions $F/k(T)$ one tries to realize (ideally over $\Qq$). Marked G-covers defined over $k$ are those G-covers $f:X\rightarrow \Pp^1_k$ with a \emph{totally $k$-rational fiber above $\infty$} (i.e. only consisting of $k$-rational points); they correspond to $k$-regular extensions $F/k(T)$ \hbox{with some $k(T)$-embedding $F \hookrightarrow k((1/T))$.}
\end{remark}

\subsubsection{Main result} 
For both categories of unmarked or marked G-covers, the main result states that
\emph{the stack of G-covers of $\Pp^1_\Cc$ of group $G$ and with $r$ branch points admits a coarse\footnote{And in fact \emph{fine} for marked G-covers; see Proposition \ref{prop:basic-descent}.} moduli space ${\rm Hur}_{r,G}$ defined over $\Qq$ with compatible action of Galois}. In particular:

\begin{theorem} \label{thm-hurwitz-space}  There is a union ${\rm Hur}_{r,G}$ of smooth and quasi-projective \hbox{geometrically} irreducible varieties defined over $\overline \Qq$, such that ${\rm Hur}_{r,G}$ is globally invariant under ${\rm Gal}(\overline \Qq/\Qq)$, and has the following properties.
\vskip 0,5mm

\noindent
For every field $k$ of characteristic $0$ with algebraic closure $\overline k$,
\vskip 0,5mm

\noindent
{\rm (a)} there is a bijection between the set ${\rm Hur}_{r,G}(\overline k)$ of $\overline k$-rational points on ${\rm Hur}_{r,G}$ and the set also previously denoted ${\rm Hur}_{r,G}(\overline k)$ of isomorphism classes of G-covers of $\Pp^1_{\overline k}$ of group $G$ and with $r$ branch points in $\Aa^1(\overline k)$,
\vskip 0,5mm

\noindent
{\rm (b)} for every G-cover $f$ of $\Pp^1_{\overline k}$ of group $G$ and with $r$ branch points in $\Aa^1(\overline k)$, 
if $[f]$ is the representative point on the variety ${\rm Hur}_{r,G}$, we have  $[f^\tau] = [f]^\tau$, for every $\tau \in {\rm Gal}(\overline k/k)$.
\vskip 1mm

\noindent
Furthermore, the map $\Phi: {\rm Hur}_{r,G}(\Cc) \rightarrow {\rm Conf}_{r}(\Cc)$ that sends each point $[f]\in {\rm Hur}_{r,G}(\Cc)$ to the branch set $\underline t$ of some/any representing cover $f$ of the isomorphism class $[f]$ is a finite topological (unramified) cover, and is induced by a finite \'etale morphism ${\rm Hur}_{r,G} \rightarrow {\rm Conf}_{r}$ between the underlying algebraic varieties.
\end{theorem}

\noindent
{\it Comments on proof}. The proof follows a similar pattern as the Riemann-Hurwitz construction. 
As we already know, the map $\Phi: {\rm Hur}_{r,G}(\Cc) \rightarrow {\rm Conf}_{r}(\Cc)$ is a finite unramified analytic cover of smooth analytic spaces. The configuration space ${\rm Conf}_{r}(\Cc)$ being an algebraic variety defined over $\Cc$, one can appeal to classical GAGA results of Serre and Grauert-Remmert
\cite{serreGAGA}, \cite{G-R1}, \cite{G-R2}, \cite{G-R3}. These generalize the completion and algebraization arguments mentioned in the one dimensional Riemann-Hurwitz construction (\S \ref{ssec:RH-construction}) to show that the cover $\Phi$ comes from an algebraic finite \'etale cover ${\rm Hur}_{r,G} \rightarrow {\rm Conf}_{r}$ defined over $\Cc$. A final argument involving Weil's descent shows that it 
can be defined over $\Qq$. 
For more details, see \cite{fried77}, \cite{FrVo1}, \cite{Vo96}.

\subsubsection{Topological description} \label{ssec:top-description} \hskip 1mm
\vskip 1mm

\noindent
\ref{ssec:top-description}.1. {\it The braid group action.} The map $\Phi: {\rm Hur}_{r,G}(\Cc) \rightarrow {\rm Conf}_{r}(\Cc)$ is originally a topological (unramified) finite cover. Classically the monodromy attaches to it an action on any fiber of the fundamental group of the base space ${\rm Conf}_{r}(\Cc)$. This fundamental group is  known to be the {\it braid group $B_r$ on $r$ strands}. Furthermore, any fiber is  in one-one correspondence with the set of possible isomorphism classes $\overline \varphi$ of homomorphisms $\varphi \in {\rm Hom}(\pi_1(\Pp^1(\Cc)\setminus \underline t,\infty),G)$, or equivalently, via the choice of a bouquet $\underline \gamma$ for $\Pp^1(\Cc)\setminus \underline t$ based at $\infty$, with the set
\vskip 1mm

\centerline{${\rm Ni}_{r,G} = \{{\underline g} = (g_1,\ldots,g_r) \in G^r \hskip 2pt | \hskip 2pt g_1\cdots g_r = 1\} / \sim$ ,}
\vskip 1mm

\noindent
called the \emph{Nielsen class}, of all possible ``branch cycle descriptions'' $\underline g$ modulo $\sim$. The symbol ``$\sim$'' indicates that the tuples $\underline g$ should be regarded up to the appropriate equivalence corresponding to the marked or unmarked situations. For simplicity we often omit the reference to $\sim$ and use usual tuples $(g_1,\ldots,g_r)$ to denote elements in the Nielsen class.

The action of the braid group $B_r$ on the Nielsen class ${\rm Ni}_{r,G}$ can be explicitly computed. Namely, the group $B_r$ has standard generators $Q_1,\ldots,Q_{r-1}$, which act on elements $\underline  g=(g_1,\ldots,g_r) \in {\rm Ni}_{r,G}$ as follows:
\vskip 1mm

\centerline{$Q_i\cdot \underline g = (g_1\ldots, g_ig_{i+1}g_i^{-1}, g_i,\ldots, g_r)$,\hskip 2mm  $i=1,\ldots,r-1$.}
\vskip 2mm

\noindent
\ref{ssec:top-description}.2. {\it Components of Hurwitz space.} \label{ssec:comp-Hurwitz}
This explicitness allows a thorough study of the topological properties of the cover $\Phi: {\rm Hur}_{r,G}(\Cc) \rightarrow {\rm Conf}_{r}(\Cc)$. In particular, we have this
\vskip 2mm

\noindent
{\bf Addendum to Theorem \ref{thm-hurwitz-space}}. {\it The components of the Hurwitz space ${\rm Hur}_{r,G}(\Cc)$ are in one-one correspondence with the orbits of  the action of the braid group $B_r$ on the Nielsen class ${\rm Ni}_{r,G}$. They coincide with the geometrically irreducible components of the Hurwitz space ${\rm Hur}_{r,G}$. For simplicity we call them the \emph{components} of ${\rm Hur}_{r,G}$.}
\vskip 2mm

Given an element $\underline g$ in the Nielsen class ${\rm Ni}_{r,G}$, denote its orbit under the braid group $B_r$ by $ B_r \hskip 1pt \underline g$. The corresponding component of the Hurwitz space ${\rm Hur}_{r,G}(\Cc)$ can be described as follows \cite[\S 1.1]{DeEm04}:

- it is the set of points $(\underline t,\varphi) \in {\rm Hur}_{r,G}(\Cc)$ 
such that $(\varphi(\gamma_1),\ldots,\varphi(\gamma_r))\in B_r \hskip 1pt \underline g$ for some/any bouquet $\underline \gamma = (\gamma_1,\ldots,\gamma_r)$ for $\Pp^1(\Cc)\setminus \underline t)$ based at $\infty$.

\noindent
or, equivalently, in terms of covers, 

- the set of all topological covers $f:X\rightarrow \Pp^1(\Cc)\setminus \underline t$ with $\underline t \in {\rm Conf}_r(\Cc)$ and such that for some topological bouquet $\underline \gamma$ of $\Pp^1(\Cc)\setminus \underline t$ based at $\infty$, we have ${\rm BCD}_{\underline \gamma}(f) = \underline g$, 

\noindent
or also as follows:

- if a  bouquet $\underline \gamma_{\underline t}$ of $\Pp^1(\Cc)\setminus \underline t$ based at $\infty$ is fixed for each $\underline t \in {\rm Conf}_r(\Cc)$, it is the set of all topological covers $f:X\rightarrow \Pp^1(\Cc)\setminus \underline t$ such that ${\rm BCD}_{\underline \gamma_{\underline t}}(f) \in B_r \hskip 1pt \underline g$, for all $\underline t \in {\rm Conf}_r(\Cc)$.

\begin{notation} We denote the component of the Hurwitz space corresponding to the orbit $B_r \hskip 1pt \underline g$ by $[\underline g]$ and sometimes abuse notation to also write $[\underline g]$ for $B_r \hskip 1pt \underline g$ itself. This identification does not depend on the choice of a topological bouquet.
\end{notation}

\begin{remark}[Marked and unmarked Hurwitz spaces] \label{rem:marked-unmarked}
There is a natural map from the marked Hurwitz space to the unmarked Hurwitz space. Let $H\subset G$ be  subgroup. Via this map, a marked G-cover of group $H$ is mapped to its unmarked version (of group $H$). If $\mathcal H$ is a component of the unmarked Hurwitz space ${\rm Hur}_{r,H}$ parametrizing \emph{connected} G-covers (of group $H$), its inverse image $\mathcal H^\ast$ is a component of the marked Hurwitz space ${\rm Hur}_{r,G}$. The reason is the following (see \hbox{e.g.} \cite[Prop.3.3.11(v)]{these-seguin}):
\vskip 0,5mm

\noindent
(*) {\it for connected covers, the inner automorphisms of $G$ can be realized by braids.} 
\vskip 0,5mm

\noindent
Consequently, the field of definition of $\mathcal H^\ast$ is the same as that of $\mathcal H$. For a connected G-cover of group $H$, working with its component in the unmarked Hurwitz space ${\rm Hur}_{r,H}$  or in the marked Hurwitz space ${\rm Hur}_{r,G}$
is essentially the same.
\end{remark} 

\subsection{Hurwitz descent from $\Cc$ to $\Qq$} \label{sec:descent}
The way the Hurwitz space is used to detect whether some of the ``many''  constructed Galois extensions $N/\Cc(T)$ of group $G$ can be ``descended to $\Qq$'' goes as follows. Let $\mathcal H\subset {\rm Hur}_{r,G}(\Cc)$ be a component of the Hurwitz space ${\rm Hur}_{r,G}$. Assume further that:

\noindent
(a) \emph{the component $\mathcal H$ is defined over $\Qq$},

\noindent
(b) \emph{the component $\mathcal H$ parametrizes \emph{connected} G-covers}.
\vskip 1mm

For each point $h\in \mathcal H$, denote its field of definition by $\Qq(h)$. 
The function field extension of the corresponding $\overline{\Qq(h)}$-G-cover is a Galois extension $N_h/\overline{\Qq(h)}(T)$ of group $G$ that is fixed by the action of the group ${\rm Gal}(\overline{\Qq(h)}/\Qq(h))$. One traditionally rephrases the last condition by saying that the {\it field of moduli} of the extension $N_h/\overline{\Qq(h)}(T)$ is $\Qq(h)$.
\vskip 0,5mm

Assume first that we work with marked G-covers. As recalled in \S \ref{ssec:descent}, there is then \emph{no obstruction to the field of moduli being a field of definition}: there exists a Galois extension $N_{h,0}/\Qq(h)(T)$, regular over $\Qq(h)$ (\hbox{i.e.} $N_{h,0} \cap \overline{\Qq(h)} = \Qq(h)$), 
that gives the extension $N_h/\overline{\Qq(h)}(T)$ by base change from $\Qq(h)$ to $\overline{\Qq(h)}$. 
Consequently ${\rm Gal}(N_{h,0}/\Qq(h)(T)) = G$. We obtain this conclusion:

\begin{proposition} \label{prop:Riemann-Hurwitz-descent}
In order to find a Galois extension $N_0/\Qq(T)$ of group $G$ regular over $\Qq$, it suffices to find,  for some $r\geq 2$, a $\Qq$-rational point on some component $\mathcal H\subset {\rm Hur}_{r,G}$ defined over $\Qq$ and parametrizing connected marked covers. 
\end{proposition} 

\noindent
Furthermore, the extension $N_0/\Qq(T)$ has a totally $\Qq$-rational fiber above $\infty$ (Remark \ref{rem:marked-unmarked1}). 
\vskip 1mm

If we work with unmarked G-covers, there may be an obstruction to the field of moduli being a field of definition; see \S \ref{ssec:descent}. This obstruction is however well-understood and known to vanish in some remarkable circumstances \cite{DeDo1}. For example, it does if $G$ is centerless. Under this extra assumption, the same conclusion as in Proposition \ref{prop:Riemann-Hurwitz-descent} holds.\footnote{Note though that in this context, the fiber above $\infty$ need not be totally $k$-rational.}
 The centerless assumption is in fact not restrictive towards the Regular Inverse Galois Problem: each finite group is known to be the quotient of some centerless group.

\begin{example} In the classical example that the group $G$ is the Fisher-Griess Monster group and $r=3$ (\cite[\S 7.4.7]{Serre-topics}, \cite{thompson}), there is a component $\mathcal H\subset {\rm Hur}_{3,G}$ of the Hurwitz space satisfying assumptions {\rm (a)} and {\rm (b)} and such that the morphism
$\Phi: \mathcal H\rightarrow {\rm Conf}_{3,G}$ is an isomorphism. Thus many $\Qq$-rational points exist on $\mathcal H$ and 
the centerless group Monster group is a regular Galois group over $\Qq$.
\end{example}

As is well-known, this example was the beginning of a series of regular realizations of groups over $\Qq$ (many of them simple) that revived the topic in the eighties and nineties. The strategy was the same, the main difficulty being to eventually find rational points on the Hurwitz space. For some groups, it could be overcome thanks to more and more sophisticated methods. But the problem remains open for many groups, including the Mathieu simple group $M_{23}$. Fried, Malle, Matzat, Thompson, V\"olklein and their schools were the main contributors. We refer to their papers and books, \emph{e.g.} \cite{Vo96}, \cite{MalMat}.

\subsubsection{Inertia Canonical Invariant} \label{ssec:ici} Beside the branch point set and the monodromy group, another classical invariant of G-covers is the \emph{Inertia Canonical Invariant}. This invariant is the {\emph{unordered} tuple $(C_1,\ldots,C_r)$\footnote{That is, a $r$-tuple modulo the action of the symmetric group $S_r$.} whose entries are the conjugacy classes $C_i$ of the branch cycles $g_i$ in the monodromy group of the G-cover (\hbox{w.r.t.} some/any bouquet)}. 

Classically, these conjugacy classes are also those of some ``distinguished'' generators of the nontrivial \emph{inertia} groups.
For complex $G$-covers, these generators are those mapping $(T-t_i)^{1/e_i}$ to $e^{2i\pi/e_i}(T-t_i)^{1/e_i}$ (where $e_i$ is the ramification index associated with the branch point $t_i$).

This invariant is constant on components $\mathcal H\subset {\rm Hur}_{r,G}$ (due to its topological nature).  We denote it by ${\rm ICI}(\mathcal H)$. 
Furthermore the group ${\rm Gal}(\overline \Qq/\Qq)$ has a simple action on this invariant. Namely it follows from the \emph{Branch Cycle Lemma} (\cite{fried77}, \cite[\S 6.1.3.5]{coursM2}) that the action of any element $\tau \in {\rm Gal}(\overline \Qq/\Qq)$ raises these conjugacy classes to their cyclotomic $\chi(\tau)^{-1}$-power. Here $\chi: {\rm Gal}(\overline \Qq/\Qq) \rightarrow (\Zz/|G|\Zz)^\times$ is the cyclotomic character. That is:
\vskip 2mm

\noindent
(*) {\it If $\mathcal H\subset {\rm Hur}_{r,G}$ is a component of the Hurwitz space, then ${\rm ICI}(\mathcal H^\tau) = {\rm ICI}(\mathcal H)^{\hskip 1pt \chi(\tau)^{-1}}$}
\vskip 2mm

\noindent
A necessary condition for a component $\mathcal H$ to be defined over $\Qq$ is that ${\rm ICI}(\mathcal H)$ is invariant when raised to prime-to-$|G|$ powers. In this case, ${\rm ICI}(\mathcal H)$ is said to be \emph{globally $\Qq$-rational}.

\begin{notation} \label{notation_HGC} Given an unordered $r$-tuple $\underline C$ of nontrivial conjugacy classes of the group $G$, it is common to denote by ${\rm Hur}_{r,G}(\underline C)$ the union of all components $\mathcal H\subset {\rm Hur}_{r,G}$ such that ${\rm ICI}({\mathcal H})=\underline C$, and to also call them \emph{Hurwitz spaces}.
\end{notation}

\begin{remark}
The Inertia Canonical Invariant is also called \emph{multidiscriminant} by some authors. They equivalently define it as the application from the set of nontrivial conjugacy classes of $G$ to $\Nn$ that sends every conjugacy class of $G$ to the number of occurrences of it among the conjugacy classes of the branch cycles; see \hbox{e.g.} \cite[\S 2.3.6]{these-seguin}.
\end{remark}

\subsection{Patching} \label{sec:patching}
As said earlier, patching is a $p$-adic avatar of the Riemann-Hurwitz construction. The complex field $\Cc$ is replaced by the $p$-adic field $\Qq_p$, or more generally, by any complete valued field $(k,v)$ for a nontrivial nonarchimedean real-valued valuation $v$.  A famous application is the solution of the Regular Inverse Galois Problem over complete valued fields (see below) and more generally over \emph{large} fields (Section \ref{ssec:RIGP-large}).  

\subsubsection{Harbater's patching result} \label{ssec:patching}
Recall that an extension $N/k(T)$ has a \emph{totally $k$-rational} fiber above some point $t_0\in \Pp^1(k)$ if $N\subset k((T-t_0))$ (and $N\subset k((1/T))$ if $t_0=\infty$).

\begin{theorem} \label{thm:patching} Let $G$ be a finite group, $G_1,G_2$ be two subgroups of $G$. Suppose given, for $i=1,2$, a Galois extension $F_i/k(T)$, of group $G_i$, regular over $k$ and with some totally $k$-rational fiber. 
Then there exists a Galois extension $F/k(T)$  
regular over $k$, of group the subgroup $\langle G_1,G_2\rangle \subset G$ generated by $G_1,G_2$ and with some totally $k$-rational fiber.
\end{theorem}

Based on the classical fact that cyclic groups are Galois groups of extensions $F/k(T)$, regular over $k$ and with some totally $k$-rational fiber, we obtain this $p$-adic analog of the Inverse Galois Problem over $\Cc(T)$.
\vskip 2mm

\noindent
{\bf Regular Inverse Galois Problem over $\Qq_p$.} {\it For $k$ as above and typically for $k=\Qq_p$, every finite group is the group of some Galois extension of $k(T)$, regular over $k$ and with some totally $k$-rational fiber.}
\vskip 2mm

Theorem \ref{thm:patching} can be stated more geometrically by replacing everywhere the phrase ``Galois extension of $k(T)$, regular over $k$'' by ``connected $k$-G-cover of $\Pp^1$''. 

\begin{proof}[Comments on proofs]
Theorem \ref{thm:patching} is due to Harbater \cite{Harbater87}. His proof somehow reproduces the complex Riemann-Hurwitz strategy over the $p$-adics, by replacing complex analytic geometry by formal analytic geometry, GAGA theorems by analogous results of Grothendieck et al. Following an idea of Serre, Theorem \ref{thm:patching} was then reproved by Liu \cite{liu} who used rigid analytic geometry instead of formal analytic geometry. These approaches, known as {\it formal patching} and {\it rigid patching}, use the geometric environment of G-covers. They were next  recast in the pure algebraic environment of field extensions by Haran-V\"oelklein \cite{haran-voelklein}. The outcome, {\it algebraic patching}, is elementary and explicit. For example, the GAGA theorems come down to a matrix factorization lemma due to Cartan (which is the main place using completeness of the field $k$); and all constructions can be performed inside the fraction field of the ring
\vskip 1mm

\hskip 20mm $\displaystyle k\{T,T^{-1}\} = \left\{\sum_{n=-\infty}^{+\infty} a_n T^n \hskip 2pt \middle\vert \hskip 2pt a_n \in k \hskip 2mm (n\in \Zz) \hbox{ and } \lim_{|n|\rightarrow \infty} a_n = 0\right\}$. \hskip 20mm \end{proof}

\begin{remark} (a) Following Liu, Pop \cite{PopLarge} developed rigid patching and obtained a more general version of the result for \emph{embedding problems} (and not just realization problems). This led to his ``1/2-Riemann Existence Theorem'', which is a weak analog of the presentation as a free group of the topological fundamental group of a punctured sphere. 
\vskip 1mm

\noindent
(b) As shown in \cite{fehm-haran-paran}, the patching approach can also be used over $\Cc$ 
to give a simple self-contained proof of the Inverse Galois Problem over $\Cc(T)$.
\end{remark}

\subsubsection{The specific patching operation} With notation as in Theorem \ref{thm:patching}, assume that the extension $F_i/k(T)$ has a totally rational fiber above some point $t_{i0}\in \Aa^1(k)$ ($i=1,2$) and that $t_{10}\not= t_{20}$; possibly compose $T$ with some M\"obius transformations to fulfill this assumption. 
Suppose also given a primitive element ${\mathcal Y}_i$ of the extension $F_i/k(T)$, $i=1,2$.

Compose then $T$ in each extension $F_i/k(T)$ with this M\"obius transformation: 
\vskip 1mm
\centerline{$\displaystyle T \rightarrow \frac{t_{01}c_iT+t_{02}}{c_iT+1}$ with $c_i\in k$,\hskip 3mm for $i=1,2$.}
\vskip 1mm

This yieds a new extension $F_i/k(T)$ that has a totally rational fiber above $\infty$ for $i=1$, and above $0$ for $i=2$. Furthermore, if $|c_1|_v$ is suitably large and $|c_2|_v$ is suitably small, 
\vskip 0,5mm

\noindent
(a) the Laurent series  in $k((1/T))$ representing ${\mathcal Y}_1$ and its $k(T)$-conjugates ${\mathcal Y}_{1i}$ (resp.  the Laurent series  in $k((T))$ representing ${\mathcal Y}_2$ and its $k(T)$-conjugates ${\mathcal Y}_{2i}$) have a radius of convergence $>1$,

\noindent
(b) the nonzero differences ${\mathcal Y}_{1i}- {\mathcal Y}_{1j}$ and ${\mathcal Y}_{2i}- {\mathcal Y}_{2j}$ are invertible in the ring $k\{T,T^{-1}\}$.
\vskip 1mm

It follows from these conditions that $F_2/k(T)$ has totally rational fibers above all $t$ 
in an open disk $D_2$ containing the closed unit disk $\overline{D(0,1)}$. Consequently the branch points of $F_2/k(T)$ lie in the open unit disk $D(\infty,1)$.
Similarly $F_1/k(T)$ has totally rational fibers above all
$t$ in an open disk $D_1 \supset \overline{D(\infty,1)}$; hence, its branch points lie in $D(0,1)$.
\vskip 1mm

The patching operation attaches to this situation an explicit extension $F/k(T)$ satisfying the requirements of Theorem \ref{thm:patching}. See in particular \cite{haran-voelklein} that produces a specific Galois extension $F/k(T) \subset {\rm Frac}(k\{T,T^{-1}\})$ along with a group action $\varphi: G\rightarrow {\rm Aut}(F)$ and uses Cartan's lemma to show the deeper part of the argument that the action $\varphi$ is faithful. The final part shows that $F/k(T)$ has a totally $k$-rational fiber, and hence is regular over $k$. 
\vskip 1mm

Furthermore, the construction provides this additional conclusion.
\vskip 1mm

\noindent
{\bf Addendum to Theorem \ref{thm:patching}.} {\it The inertia canonical invariant of the extension $F/k(T)$ obtained by patching $F_1/k(T)$ and $F_2/k(T)$ is the disjoint union (or concatenation) of the inertia canonical invariants of $F_1/k(T)$ 
and $F_2/k(T)$.}

\begin{proof}[Comments on proof] What should be used from the construction of $F/k(T)$ is that if $t$ is a branch point of $F/k(T)$, then $|t|_v \not=1$, and
\vskip 0,5mm

\noindent
- either $|t|_v < 1$ and $t$ is not a branch point of $F_2/k(T)$, but is a branch point of $F_1/k(T)$, with same inertia groups (in fact $F k((T-t)) = F_1 k((T-t))$ up to $k(T)$-isomorphism), 
\vskip 0,5mm

\noindent
- or $|t^{-1}|_v < 1$ and $t$ is not a branch point of $F_1/k(T)$, but is a branch point of $F_2/k(T)$, with same inertia groups (in fact $F k((T-t)) = F_2 k((T-t))$ up to $k(T)$-isomorphism),
\vskip 0,5mm

\noindent
(If $|t|_v =1$, $t$ is not a branch point of $F/k(T)$; and, up to $k(T)$-isomorphism, we have $F k((T-t_0)) = F_1 k((T-t_0))= F_2 k((T-t_0)) = k((T-t_0))$.
\end{proof}

\begin{remark} Denote the branch point number of $F_i/k(T)$ by $r_i$ and the point on the marked Hurwitz space ${\rm Hur}_{r_i,G_i}$ representing $F_i/k(T)$ by $h_i$, $i=1,2$. Not all couples $(h_1,h_2)\in {\rm Hur}_{r_1,G_1}(k)\times {\rm Hur}_{r_2,G_2}(k)$ correspond to $k$-G-extensions
that can be patched. Their branch points should be in some special position, which the preliminarily M\"obius transformations make it possible to reach. Patching does not naturally induce a map 
\vskip 0,5mm

\centerline{${\rm Hur}_{r_1,G_1}\times {\rm Hur}_{r_2,G_2}\rightarrow {\rm Hur}_{r_1+r_2,\langle G_1,G_2\rangle}$.}
\vskip 0,5mm

\noindent
On the other hand, once the special position of the branch points is fulfilled, the patched G-extension $F/k(T)$ is quite precisely and algebraically defined from $F_1/k(T)$ and $F_2/k(T)$. 
\end{remark}

\subsection{Non-Galois context} \label{ssec:nonGalois}
This paper focusing on the use of Hurwitz spaces in inverse Galois theory,
we have restricted to Galois covers. It is however worth recalling that there is a Hurwitz space theory for not necessarily Galois covers as well. The main technical difference is that, for connected covers (for simplicity), the group $G$ is given with an embedding $G\rightarrow S_d$ (not necessarily the regular representation) and the surjective morphisms $\pi_1(\Pp^1\setminus \underline t,\infty) \rightarrow G$ should be replaced by transitive actions $\pi_1(\Pp^1\setminus \underline t,\infty) \rightarrow S_d$. The notation ${\rm Hur}_{r,d}$ replaces the notation  ${\rm Hur}_{r,G}$.
\vskip 1mm

The non-Galois context was that of Hurwitz original motivation, which was to prove \emph{the connectedness of the moduli space ${\mathcal M}_g$ of curves of genus $g$}. 
The argument goes as follows.

Algebraic considerations show that each complex genus $g\geq 2$ curve can be presented as a simply ramified cover of degree $d$ of $\Pp^1$, if $d\geq g+1$ (``simply'' meaning that all ramified fibers are of cardinality $d-1$). This 
gives rise to a continuous surjective map 
\vskip 0,5mm

\centerline{${\rm Hur}_{r,d}(2^r) \rightarrow {\mathcal M}_g$,} 
\vskip 0,5mm

\noindent
where ${\rm Hur}_{r,d}(2^r)$ is the subspace of ${\rm Hur}_{r,d}$ of all covers $X\rightarrow \Pp^1$ of monodromy group $S_d$ and $r$ branch points with $2$-cycles as branch cycles, for some integer $r=r(g,d)$.
The crucial point is that ${\rm Hur}_{r,d}(2^r)$ is connected as a consequence of the transitivity of the appropriate braid group action. Connectedness of ${\mathcal M}_g$ follows. This was improved by Severi, Fulton and Deligne-Mumford to show that ${\mathcal M}_g$ is an irreducible variety in all cha\-rac\-te\-ristics. A similar argument shows irreducibility of modular curves. See \cite[\S 2.4]{DeZakopane}.
\vskip 1mm

The general natural map ${\rm Hur}_{r,d} \rightarrow {\mathcal M}_g$ (where $r$, $d$ and $g$ are related by the Riemann-Hurwitz formula) is important for it connects two different worlds: covers have a topological origin while curves are algebraic objects. Questions about this map may be difficult (\hbox{e.g.} Question \ref{question-Mochizuki}  below) but rewarding. For example, works of V\"olklein \cite{MR2148464} on special loci of  ${\mathcal M}_g$ in connection with some of Cornalba \cite{MR932781} rely on this map. In a more arithmetic direction, the knowledge we have on the action
of ${\rm Gal}(\overline{\mathbb Q}/{\mathbb Q})$ on Hurwitz spaces can yield information on special loci of ${\mathcal M}_g$ and its stack structure; see \cite{collas-maugeais1} \cite{collas-maugeais2} and the anabelian \cite{tsujimura}. Another map from the Hurwitz space connects two different worlds, in the Galois context this time: it is the map ${\rm Hur}_{r,G} \rightarrow {\mathcal N}_G$ to the Noether variety ${\mathcal N}_G = \Aa^d/G$ with $d=|G|$, which lands in the territory of rationally connected varieties.

\begin{question} \label{question-Mochizuki} The following question was asked to me by S.~Mochizuki. In the Hurwitz original situation, the map ${\rm Hur}_{r,d}(2^r) \rightarrow {\mathcal M}_g$ is dominant, and 
it is generically non-quasi-finite (for dimension reasons).
On the contrary, in the case $d=2$, the map ${\rm Hur}_{r,d} \rightarrow {\mathcal M}_g$ is not dominant if 
$g\geq 3$ (not all curves are hyperelliptic); and it is quasi-finite if $g\geq 2$ (as degree $2$ covers $X \rightarrow\Pp^1$ are 
Galois and come from the finite list of auto\-mor\-phisms of $X$). The question is:
to what extent are there examples of components $\mathcal H\subset {\rm Hur}_{r,d}$ such that $\mathcal H \rightarrow {\mathcal M}_g$ is not dominant, but nevertheless generically non-quasi-finite?
\end{question}

Hurwitz spaces of not necessarily Galois covers may be of interest in an arithmetic per\-spec\-tive. That is the case when one looks for covers of $\Pp^1$ satisfying certain ramification constraints (on the branch points, on the inertia groups) and that are defined over $\Qq$ (or some other fixed field). The question reduces to finding rational points on some appropriate Hurwitz space.
Examples are given in  \cite{DeFr1}, \cite{DeFr4}; see also \cite[\S 4]{DeZakopane}.

\section{Semi-modern developments} \label{sec:further-developments}
This section discusses the following topics, which have become solid parts of Hurwitz space theory: rational points over large fields (Section \ref{ssec:RIGP-large}), integral scheme models, compactification (Section \ref{ssec:Hurwitz-scheme}), descent theory (Section \ref{ssec:descent}), modular towers (Section \ref{ssec:modular-towers}).

\subsection{Large fields} \label{ssec:RIGP-large}
Since Hurwitz spaces and patching have shown their potential for inverse Galois theory in the eighties, these two theories have developed side by side in a parallel way and have also often met. This section shows a first example.
\vskip 1mm

A main point is that $k$-G-covers obtained by patching provide $k$-rational points on appropriate Hurwitz spaces. The next idea is to use them for finer arithmetic conclusions, such as identifying points or components defined over small fields.
First such investigations emerged in the 1990s. Papers \cite{DeFr3} (for the real case $k=\Rr$) and \cite{DeSeattle} (for the $p$-adic case $k=\Qq_p$), conjoined with \cite{BrunoSeattle}, show that not only every finite group $G$ is the group of some $k$-G-extension 
$F/k(T)$, but that also the following is true.
\vskip 1mm

\begin{theorem} \label{thm:thm-Qtr-Qtp} For every finite group $G$, one can construct an integer $r\geq 2$ and a component $\mathcal H\subset {\rm Hur}_{r,G}$ that is \emph{defined over $\Qq$} and has $\Rr$-rational points (when $k=\Rr)$ and $\Qq_p$-rational points
(when $k=\Qq_p$).
\end{theorem}

\begin{proof}[{Comments of proof of Theorem \ref{thm:thm-Qtr-Qtp}}] The proof combines several ideas. 

A first one is that a necessary condition for a component $\mathcal H$ to be defined over $\Qq$ is that the invariant ${\rm ICI}(\mathcal H)$ be globally $\Qq$-rational (\S \ref{ssec:ici}). This can easily be \hbox{achieved by} working with covers with inertia canonical invariant  $\underline C$ such that if some conjugacy class $C$ of $G$ occurs, then all prime-to-$G$ powers of $C$ occur the same number of times. 
\vskip 0,5mm 

For such an unordered tuple $\underline C=(C_1,\ldots,C_r)$ of conjugacy classes, an easy way to obtain a component $\mathcal H\subset {\rm Hur}_{r,G}(\underline C)$ defined over $\Qq$ is to guarantee that ${\rm Hur}_{r,G}(\underline C)$ itself is irreducible (and then take $\mathcal H={\rm Hur}_{r,G}(\underline C)$), \hbox{i.e.}, to guarantee that the braid group acts transitively on the subset of the 
Nielsen class with ICI equal to ${\underline C}$:
\vskip 1mm

\centerline{${\rm Ni}_{r,G}(\underline C)= \{ (g_1\ldots,g_r) \in {\rm Ni}_{r,G}\hskip 2pt | \hskip 2pt g_i\in C_{\omega(i)}\ \hbox{\rm for some } \omega \in S_r\}$.}
\vskip 1mm

\noindent
Such a transitivity property is the point of the so-called \emph{Conway-Parker-Fried-V\"olklein theorem} (see Theorem \ref{thm:CPFV} and Remark \ref{rem:CPFV}): under some assumption which one can always reduce to, transitivity holds provided that each nontrivial conjugacy class of $G$ occurs suitably often in $\underline C$. As before, it suffices to throw 
in as many suitable conjugacy classes as necessary to reach the condition.
\vskip 0,5mm

The final step is to show that the corresponding component $\mathcal H$ has real points and $p$-adic points for all $p$, \hbox{i.e.} that G-covers defined over $\Rr$ and all $\Qq_p$ can be constructed with group $G$ and 
inertia canonical invariant
equal to the constructed tuple $\underline C$. For the $p$-adic case, patching is the key (Theorem \ref{thm:patching}). For the real case, one uses results on the action of complex conjugation on the complex covers produced by the Riemann-Hurwitz strategy (\cite{DFr90},\cite{DeFr3}). A slightly finer analysis (performed in \cite{BrunoSeattle}) and some adjustments on $\underline C$ are necessary to make $r, \underline C, \mathcal H$ the same for all primes $p$ and for the reals.
\end{proof}

Interestingly enough, Theorem \ref{thm:thm-Qtr-Qtp} could then be combined with some general result on rational points on varieties 
to the effect that if $V$ is a smooth geometrically irreducible variety $V$,
defined over $\Qq$, with real points (\hbox{resp.} $p$-adic points), then $V$ has totally real points (\hbox{resp.} totally $p$-adic points). Such a result appears in a paper of Moret-Bailly \cite{Moret-Bailly-SkolemII}, which improves on works of Cantor-Roquette, Rumely, Green-Pop-Roquette.

We obtain the following result. We denote by $\Qq^{\rm tr}$ (\hbox{resp.} $\Qq^{{\rm t}p}$)
 the field of \emph{totally real algebraic numbers}, \hbox{i.e.} whose all conjugates are real (\hbox{resp.} \emph{totally $p$-adic algebraic numbers}, \hbox{i.e.} whose all conjugates are $p$-adic).

\begin{corollary}[\cite{DeFr3}, \cite{DeSeattle}] \label{RIGP-large}
With $k= \Qq^{\rm tr}$ or $k=\Qq^{{\rm t}p}$, every finite group is the group of some Galois  extension of $k(T)$, regular over $k$, with some totally $k$-rational fiber.
\end{corollary}

These fields are the ``smallest'' fields of characteristic $0$ over which the regular Inverse Galois Problem 
is known to hold.

\begin{remark}[Harbater-Pop \hbox{vs.} D\`ebes-Fried] (a) Corollary  \ref{RIGP-large} was generalized by Pop \cite{PopLarge} to the case that $k$ is an {\it ample} field, \hbox{i.e.}, a field $k$ with the property that every smooth geometrically irreducible curve defined over $k$ has infinitely many $k$-rational points provided it has at least one (Pop calls them ``large'' instead of ``ample'').  Ample fields do include $\Qq^{\rm tr}$ and $\Qq^{{\rm t}p}$: this is in fact what 
Moret-Bailly shows. Other examples are separably closed fields, complete or henselian valued fields, PAC fields. 
\vskip 1mm

\noindent
(b) Pop's approach is different in that he uses patching but not Hurwitz spaces. 
Specifically he uses patching to solve the regular Inverse Galois Problem over the complete field $k((x))$, then
interprets this as producing, for every finite group $G$, a family of G-covers parametrized by a geometrically irreducible variety $V$ such that $k(V)\subset k((x))$ and finally concludes, from the definition of ample fields, that $V$ has many $k$-rational points. An equivalent characterization of ample fields is indeed that they are  existentially closed in $k((x))$. Thus he can conclude that many of the G-covers of the family are in fact $k$-G-covers. For more details on this descent argument, see  \cite[\S 4.2]{DeDes1}.
\vskip 1mm

\noindent
(c) The Hurwitz space approach can also be used to show the general case of Corollary \ref{RIGP-large} that $k$ is an ample field (of characteristic $0$). Namely, using patching over the field $k((x))$ rather than $\Qq_p$, the component $\mathcal H$ from the proof of Theorem \ref{thm:thm-Qtr-Qtp} can be constructed so to have $k((x))$-points. Hence it has $k$-rational points if $k$ is existentially closed in $k((x))$. 
\vskip 1mm

\noindent
(d) Pop's result is more generally concerned with {\it embedding problems} and not just realization problems. It is unclear whether this generalization can be handled with the Hurwitz space approach although Fried-V\"olklein indeed do it in \cite{FV92} but only over PAC fields.
\end{remark}

\subsection{Integral scheme models and compactification} \label{ssec:Hurwitz-scheme}
Hurwitz spaces can further be defined as schemes and can be compactified. We provide the main conclusions, which are due to Wewers \cite{WewersThesis}, \cite{WewersGainesville} and Romagny-Wewers \cite{RoWe}. They have notably been used in \cite{DeEm04}, \cite{cau2012,cau2016} where a longer introduction to them is given.

\begin{theorem} \label{thm:wewers-romagny}
There exists a smooth scheme $\mathcal{H}ur_{r,G}$ of finite type over $\Zz[1/|G|]$ and a normal compactification 
$\overline{\mathcal{H}ur_{r,G}}$ such that:
\vskip 0,5mm

\noindent
{\rm (a)} $\mathcal{H}ur_{r,G} \otimes_{\Zz[1/|G|]} \Cc = \hbox{\rm Hur}_{r,G}$, 
\vskip 1mm

\noindent
and,  for each prime $p \not |  \hskip 3pt |G|$,
\vskip 0,3mm

\noindent
{\rm (b)} the fiber $\mathcal{H}_p=\mathcal{H}ur_{r,G} \otimes_{\Zz[1/|G|]} {\Ff_p}$ is a smooth variety and its $\overline{\Ff_p}$-components correspond to those of $\hbox{\rm Hur}_{r,G}$ reduced modulo $p$,
\vskip 0,3mm

\noindent
{\rm (c)} $\mathcal{H}_p$ is a moduli space for G-covers of $\Pp^1$ with $r$ branch points over algebraically closed fields of characteristic $p$.
\vskip 0,8mm

\noindent
Furthermore,
\vskip 0,3mm

\noindent
{\rm (d)} the morphism $\hbox{\rm Hur}_{r,G} \rightarrow {\rm Conf}_r$ extends to a morphism 
\vskip 1mm
\centerline{$\overline{\mathcal{H}ur_{r,G}} \rightarrow \overline{{\rm Conf}_r}$}
\vskip 1mm
\noindent
where  ${\rm Conf}_r$ is viewed as a scheme over $\Zz$ and $\overline{{\rm Conf}_r}$ is its Harris-Mumford com\-pactification. 
\end{theorem}

The set $\overline{{\rm Conf}_r} \setminus {\rm Conf}_r$ parametrizes {\it stable marked curves of genus $0$ with a root}, 
\hbox{i.e.} trees of curves of genus $0$ with a distinguished component $T_0 \simeq \Pp^1$ (the {\it root}) and at least three marked points (including the double points) on any component but the root.} Points in  $\overline{\mathcal{H}ur_{r,G}} \setminus {\mathcal{H}ur_{r,G}}$ correspond to certain degenerate covers of these singular curves.

\vskip 2mm
The last two subsections briefly discuss two topics where Hurwitz spaces appear but go beyond pure inverse Galois theory:
they exist both in the context of Galois and non-Galois covers. 
These topics are classical and much documented in the literature. 
We refer to basic references for more details.

\subsection{Descent theory} \label{ssec:descent} 
The main point connecting descent theory and Hurwitz spaces is that for $r$ and $G$ as above, and for any field $k$ of characteristic $0$, typically for $k=\Qq$,
\vskip 1mm

\noindent
(*) {\it the field of definition of the point $[f]$ on the (not necessarily irreducible) $k$-variety ${\rm Hur}_{r,G}$, which represents the isomorphism class of a $\overline k$-G-cover $f:X\rightarrow \Pp^1$, is the \emph{\it field of moduli} of $f$ (relative to the extension $\overline k/k$).}
\vskip 1mm

\noindent
Recall that the field of moduli of $f$ is the subfield of $\overline k$ fixed by the subgroup of all $\tau \in {\rm Gal}(\overline k/k)$ such that the two G-covers $f$ and $f^\tau$ are isomorphic; and so (*) follows from Theorem \ref{thm-hurwitz-space}(b). The field of moduli is contained in every field of definition of $f$ but need not be a field of definition itself although it is in some classical situations \cite{DeDo1}.
\vskip 1mm

There is a second descent problem
attached to Hurwitz space theory. 
The question is whether there is a {\it universal} family of G-covers parametrized by the Hurwitz space ${\rm Hur}_{r,G}$, i.e., such that every family of G-covers with the same invariants $r$, $G$ is a pull-back of that family. As the field of moduli for the descent of fields of definition, the Hurwitz space is the ``most natural'' space that can parametrize a universal family of G-covers. This descent problem is a general variant of the more common question of whether the Hurwitz space, which is a coarse moduli space, is 
a {\it fine} moduli space. A basic result is

\begin{proposition} \label{prop:basic-descent} The field of moduli is a field of definition, and the Hurwitz space is a parameter space for a universal family, and even is a fine moduli space, in the following situations: {\rm (a)} marked G-covers, and, {\rm (b)}
unmarked G-covers if $G$ has trivial center.
\end{proposition}

For more refined results and for references, we refer to \cite{De-Berkeley} where these two problems are treated in parallel and concrete criteria and related results are given. 

\subsection{Modular towers} \label{ssec:modular-towers} 
Modular towers, which were introduced by M.~Fried, constitute a vertical development of the Hurwitz space theory. Main Fried's references are \cite{Fr_introMT} and {\cite{BaFr}; \cite{DeLy04} and \cite{DeMT} provide a detailed introduction to the topic. 
\vskip 1mm
 
 A modular tower is a tower of moduli Hurwitz spaces $({\rm Hur}_{r,G_n}({\underline C}_n))_{n\geq 0}$ (with maps
going down) where the branch point number $r\geq 3$ is fixed and $(G_n,{\underline C}_n)_{n\geq 0}$ is a projective system  of groups $G_n$ given with an $r$-tuple ${\underline C}_n$ of conjugacy classes of order
prime to some prime divisor $p$ of $|G|$.
\vskip 1mm

The motivating example is the tower $(Y^1(p^n))_{n\geq 1}$ of modular curves. More specifically take $G_n=D_{p^n}$ the dihedral group of order $2p^n$, with $p\not=2$, and ${\underline C}_n$ the $4$-tuple consisting of the involution class of $D_{p^n}$, repeated $4$ times. A classical argument (\hbox{e.g.} \cite[\S 1.3]{DeLy04}) shows that there is a surjective morphism, defined over $\Qq$, 
\vskip 1mm

\centerline{$({\rm Hur}_{4,D_{2p^n}}(\underline C_n))_{n\geq 1} \rightarrow (Y^1(p^n)\setminus \{\hbox{\rm \small{cusps}}\})_{n\geq 1}$.}
\vskip 1mm

Fried's generalized setting starts with a fixed finite group $G$, a prime divisor $p$ of $|G|$ and
an unordered $r$-tuple ${\underline C}_1$ of conjugacy classes of $G$ of prime-to-$p$ order. The levels $G_n$ then correspond to the \emph{characteristic quotients} of the \emph{universal $p$-Frattini cover $s: {}_p \widetilde G\rightarrow G$}. That is: the group homomorphism $s$ is surjective, is Frattini (for any subgroup $H^\prime \subset H$, if  $s(H^\prime)= G$ then $H^\prime = H$), its kernel is a $p$-group, and $s$ is ``universal'' \hbox{w.r.t.} these properties. The levels $G_n$ are the quotients ${}_p \widetilde G/\widetilde P_n$, where the groups $\widetilde P_n$ are inductively defined by
\vskip 1mm

\centerline{\small $\left\{\begin{matrix} 
\widetilde P_1 = {\rm ker}({}_p \widetilde G \rightarrow G) \hfill \cr
\widetilde P_n = {\rm Fratt}(\widetilde P_{n-1}) = \widetilde P_{n-1}^p [\widetilde P_{n-1},\widetilde P_{n-1}].  \cr
\end{matrix}\right.$}

The groups $\widetilde P_n$ are pro-$p$-groups of finite rank and the levels $G_n$ are finite. In the motivating example, we have $\widetilde G = \Zz_p\rtimes \Zz/2\Zz$ and $\widetilde P_n=p^n \Zz_p$. As to the unordered tuples ${\underline C}_n$, they are obtained from $\underline C_1$ thanks to a Schur-Zassenhaus argument to the effect that each $g\in G_1=G$ of prime-to-$p$ order uniquely lifts, up to conjugation, to some element in $G_2$ with the same order, and then similarly in $G_3$, etc.
\vskip 1mm

Persistence of rational points on levels ${\rm Hur}_{r,G_n}(\underline C_n))$
is the main diophantine question on modular towers. 

\begin{conjecture}[Fried 1995] \label{conj:fried}
If $k$ is a number field, there exists an integer $n_0$ such that for every integer $n\geq n_0$, there are no $k$-rational points on ${\rm Hur}_{r,G_n}(\underline C_n)$.
\end{conjecture}

This corresponds to the
impossibility of realizing regularly all groups $G_n$ over $k$ with a bounded number of
branch points and certain ramification constraints. As noted by Fried, Conjecture \ref{conj:fried}  follows from the strong torsion conjecture on torsion points on abelian varieties over number fields \cite[\S 3.2]{CaDe}. Beyond the modular curve example, Conjecture \ref{conj:fried} was proved when $r\leq 4$ by Cadoret-Tamagawa \cite{CaTa2} but remains open if $r>4$.

\section{Recent advances I: Components defined over $\Qq$} \label{sec:components}
Following Theorem \ref{thm:thm-Qtr-Qtp} and the resulting solution of the Regular Inverse Galois Problem over ample fields, many further works, spread out over more than 25 years, have been devoted to the construction of components of Hurwitz spaces defined over $\Qq$. The gist of all approaches is to study the orbits of the braid group action. Below is a list of phrases that are used to label the main ideas introduced in these works, together with a list of contributors and references. Some papers appear several times. 
\vskip 1mm

\noindent
{\it Conway-Parker theorem}:  see Fried-V\"olklein \cite{FrVo1}, D\`ebes-Fried \cite{DeFr3}, Wood \cite{wood2021}, Seguin \cite{fielddef}.
\vskip 0,5mm

\noindent
{\it Harbater-Mumford components}:  see Fried \cite{Fr_introMT}, D\`ebes-Emsalem \cite{DeEm04}, Cau \cite{cau2012}, \cite{cau2016}, Seguin \cite{fielddef}.
\vskip 0,5mm

\noindent
{\it Lifting invariant}:  see Fried \cite{Fr_introMT}, Ellenberg-Venkatesh-Westerland \cite{ellenberg-et-al-II}, \hfill \break 
Wood \cite{wood2021}, Seguin \cite{fielddef}. 
\vskip 0,5mm

\noindent
{\it Ring of components}:  see Cau \cite{cau2012,cau2016}, Ellenberg-Venkatesh-Westerland \cite{EVW2016}, Seguin \cite{countcomp,fielddef,geomringcomp}.
\vskip 1mm

\noindent
We discuss Cau's and Seguin's results, which are the most advanced 
on $\Qq$-components.
\vskip 1mm
Finding $\Qq$-components is the first step of the search of $\Qq$-rational points on Hurwitz spaces.
Short of actually producing $\Qq$-rational points on such $\Qq$-components --which at present appears unattainable in general-- weaker developments have produced projective systems of components, defined over $\Qq$ or $\Qq^{ab}$, on certain \emph{towers of Hurwitz spaces} like modular towers, \underbar{and} projective systems of rational \emph{points} over completions of $\Qq$ or $\Qq^{ab}$ on these towers of components. For this, we refer to \cite{DeDes} and \cite{DeEm04}. A few other examples resulting from existence of $\Qq$-components on some Hurwitz spaces are given below.

\subsection{Ring of components} We start with a basic operation on components.

\subsubsection{The concatenation operation on components}  Given a finite group $G$ and an integer $r\geq 2$, denote the set of components of ${\rm Hur}_{r,G}$ by ${\rm Comp}_{r,G}$ and set 
\vskip 1mm

\centerline{$\displaystyle {\rm Comp}_G = \bigcup_{r\geq 2} {\rm Comp}_{r,G} \hskip 1,5mm \cup  \hskip 1mm \{1\}$.}
\vskip 1mm

\noindent
where $1$ is the ``Hurwitz space ${\rm Hur}_{0,G}$'' consisting of the single trivial G-cover of $\Pp^1$ (which has $r=0$ branch point). Recall that components are viewed as tuples modulo the braid action (\S \ref{ssec:comp-Hurwitz}) (and $1$ is then the $0$-tuple with components in $G$). In the context of \emph{marked G-covers}, the concatenation of tuples naturally induces a map
\vspace{1mm}

$\begin{matrix} 
\hskip 30mm {\rm Comp}_{r_1,G} \times {\rm Comp}_{r_2,G}  &\longrightarrow &{\rm Comp}_{r_1+r_2,G} \cr 
\end{matrix}$

$\begin{matrix} 
\hskip 40mm ([\underline g_1], [\underline g_2]) & \hskip 11mm \longmapsto & \hskip 4mm [\underline g_1, \underline g_2] \cr 
\end{matrix}$
\vskip 2mm

\noindent
which yields a monoid law on ${\rm Comp}_G$.
For any field $k$ of characteristic $0$, denote the set of all $k$-linear combinations of elements of ${\rm Comp}_G$ by $k[{\rm Comp}_G]$. Using the natural addition and the concatenation as multiplication, the set $k[{\rm Comp}_G]$ acquires a graded ring structure, the graduation coming from $r$. The ring $k[{\rm Comp}_G]$ is called the \emph{ring of components}. 

\begin{remark}[several situations]
The ring of components was introduced by Ellenberg-Venkatesh-Westerland for marked G-covers  of $\Aa^1$. This ring of components was subsequently used and extended to marked G-covers  of $\Pp^1$ by Seguin.
A similar ring of components had been previously introduced by Cau for unmarked \emph{connected} G-covers of $\Pp^1$; 
the concatenation is well-defined in this situation due to Remark \ref{rem:marked-unmarked}(*).
\end{remark}

\subsubsection{Scheme of components} \label{ssec:scheme-components}
It is easily checked that for two tuples $\underline g$, $\underline g^\prime$, we have: 
\vskip 0,5mm 
\centerline{$[\underline g] \times [\underline g^\prime] = [(\underline g^\prime)^{\prod_{i} g_i}] \times [\underline g]$.}
\vskip 1,5mm

\noindent 
Thus for $G$-covers of $\Pp^1$, the ring of components is \underbar{commutative}; and it is also of finite type.
The corresponding affine scheme ${\rm Spec}(k[{\rm Comp}_G])$ is called the \emph{scheme of components}.
\vskip 1mm

The scheme of components is introduced and thoroughly studied in \cite{geomringcomp}. Seguin describes it as the disjoint union of some locally closed subsets $\gamma (H)$, with $H$ ranging over subgroups of $G$. 
This stratification allows reduction to each stratum $\gamma(H)$ for many issues.
In the original Hurwitz context (\S \ref{ssec:nonGalois}), Seguin obtains this explicit description.

\begin{example} For $G=S_d$,
the subscheme of components of G-covers with branch cycles in the conjugacy class of $2$-cycles, 
%${\rm Spec}(k[{\rm Comp}_G(c)])$ 
is the affine subscheme 
of $\Aa^{d(d-1)/2}$ with equations:

\vskip 1mm
\centerline{$X_{ij}X_{jk}=X_{ik}X_{jk}=X_{ij}X_{ik}\hskip 4mm (1\leq i<j<k\leq d)$.}
\vskip 1mm

\noindent
The connectedness of ${\rm Hur}_{r,S_d}(2^r)$ originally shown by Hurwitz  corresponds to the existence of a line on this subscheme, namely 
 the line where all coordinates $X_{ij}$ are equal.
\end{example}

\subsection{Main result} Concatenation of components is a purely topological process. This 
makes unclear the answer to Question \ref{question:xy-def-K} below and quite striking the 
conclusions of the following Theorem \ref{thm:seguin2023}, due to Seguin \cite[Theorem 1.2]{fielddef}.

\begin{question} \label{question:xy-def-K}
Let $k$ be a number field and $x,y\in {\rm Comp}_G$ be two components of some Hurwitz space
of marked covers, assumed to be defined over $k$. 
Is the concatenation $x y$ of $x$ and $y$ still defined over $k$?
\end{question}

\begin{theorem} \label{thm:seguin2023}
Let $k$ be a number field and $x,y\in {\rm Comp}_G$ be components defined over $k$  
and $H, K$ the groups of the covers parametrized by $x$ and $y$. 
Then we have:

\vspace{1mm}

\noindent
{\rm (a)} If $\langle H, K \rangle = HK$, then the component $xy$ is defined over $k$.

\vspace{1mm}

\noindent
{\rm (b)} If every conjugacy class $C_i$ in the ramification type $\underline C$ of $xy$ 
appears suitably often in $\underline C$ (at least $M$ times with $M$ depending on $G$), then $xy$ is defined over $k$.

\vspace{1mm}

\noindent
{\rm (c)} There exists $\gamma \in G$ such that the component $x  y^{\gamma}$ is defined over $k$ and
\hbox{$\langle H,K^\gamma\rangle = \langle H,K\rangle$. }
\end{theorem}

In part (c), $y^\gamma$ denotes the following component: if $y=[\underline g]$,  for some tuple $\underline g$ with components in $K\subset G$, then $y^\gamma = [\underline g^\gamma]$, where the action of $\gamma$
on $\underline g$ is by componentwise conjugation. Clearly, $y^\gamma$ is defined over $k$ if and only if $y$ is.

\vskip 1mm

Part (a) is the culmination of ideas and techniques going back to the Harbater-Mumford components introduced by Fried \cite{Fr_introMT} and developed by D\`ebes-Emsalem \cite{DeEm04} and Cau \cite{cau2012}.\footnote{Cau proves a similar result as Theorem \ref{thm:seguin2023}(a) with a related but different ``completeness assumption'' replacing the assumption $\langle H,K\rangle=HK$. A more general version of Theorem \ref{thm:seguin2023}(a) with an assumption weaker than both Seguin's and Cau's assumption is given in \cite[\S 8.2.5]{these-seguin}.} A notable case of  interest is when $H$ or $K$ is normal in $\langle H, K \rangle$, \hbox{e.g.} when one is contained in the other,
and in particular when $x=y$. The proof, sketched in \S \ref{ssec:proof-part(a)}, rests on braid calculations and Cau's fundamental theorem (Theorem \ref{thm:cau} below).
\vskip 1mm

Part (b) is typical of the spirit of the Conway-Parker theorem. A modern form of the statement is Theorem \ref{thm:CPFV} below, proved in \cite{ellenberg-et-al-II} and in \cite{wood2021}. We explain how to deduce 
Theorem \ref{thm:seguin2023}(b) in \S \ref{ssec:proof-part(b)}.
\vskip 1mm

Part (c) rests on Hilbert's irreducibility theorem, patching and Cau's fundamental theorem. The following Remark \ref{rem:inductive} and Example \ref{ex:mathieu} are valuable applications of the result. A proof of Theorem \ref{thm:seguin2023}(c) is sketched in \S \ref{ssec:proof-part(c)}.

\begin{remark}[An inductive construction of $\Qq$-components] \label{rem:inductive}
The following argument shows how to inductively  use Theorem \ref{thm:seguin2023}(c) to construct, \emph{for any finite group $G$}, a Hurwitz component defined over $\Qq$ parametrizing connected G-covers of group $G$.
Pick an element $g_c$ in each nontrivial conjugacy class $c$ of $G$. For each such $g_c$, consider a component $y_c$ defined over $\Qq$ of G-covers of group $\langle g_c \rangle$ (such  components $y_c$ are 
well known to exist for cyclic groups). Then successive applications of Theorem \ref{thm:seguin2023}(c) provide elements $\gamma_c \in G$ such that the product of all components $y_c^{\gamma_c}$
is defined over $\Qq$. By a classical lemma of Jordan, the group of this component, which intersects every conjugacy class of $G$, is $G$.
\end{remark}

\begin{example}[\cite{fielddef}(Example 5.5)] \label{ex:mathieu}
The Mathieu group $G = M_{23}$ is known to be generated by two conjugate elements $a$ and $a^{\beta}$ of order $3$ ($a, \beta \in G$).  The components $x=[a, a^{-1}]$ and $x^\beta = [a^{\beta}, (a^{\beta})^{-1}]$
are defined over $\Qq$.  By Theorem \ref{thm:seguin2023}(c), for some $\gamma \in G$, the component $x\hskip 2pt x^{\gamma\beta}$ is defined over $\Qq$.  It is a component of $G$-covers of group $G = M_{23}$ of dimension $4$ defined over $\Qq$. A previous construction of Cau had $4$ replaced by $15$.
\end{example}

\subsection{Cau's theorem} We refer to \cite{cau2012} for a proof of the following result. It generalizes previous results of Fried \cite{Fr_introMT} and D\`ebes-Emsalem \cite{DeEm04} where the components $x_1,\ldots,x_n$ were components of cyclic covers branched at two points and the resulting components $x_1^{\gamma_1}\cdots x_n^{\gamma_n}$ were called Harbater-Mumford components.

\begin{theorem} \label{thm:cau}
Let $x_1, \ldots, x_n \in {\rm Comp}_G$ be components defined over $k$ and 
$H_1, \ldots, H_n$ the respective groups of the parametrized covers. 
Consider the components $x_1^{\gamma_1}\cdots x_n^{\gamma_n}$ where   
$\gamma_1, \ldots, \gamma_n$ range over the subgroup $\langle H_1, \ldots, H_n\rangle \subset G$.

\vskip 0,5mm

\noindent
{\rm (a)} These components are characterized by the 
following property: they have points corresponding to ``nice admissible covers'' on their boundary \emph{(for their Wewers compactification (\S \ref{ssec:Hurwitz-scheme}))},  \hbox{i.e.}, 
{\it covers of a ``comb with $n$ teeth''  such that the restriction above the $i$-tooth is in the component 
$x_i$ ($i=1,\ldots,n)$, and is unramified above the root of the comb}.
\vskip 1mm

\noindent
{\rm (b)} The action of ${\rm Gal}(\overline k/k)$ permutes the components 
$x_1^{\gamma_1}\cdots x_n^{\gamma_n}$ with   
$\gamma_1, \ldots, \gamma_n$ ranging over the subset of $\langle H_1\ldots, H_n\rangle^n$ such that 
$\langle H_1^{\gamma_1}, \ldots, H_n^{\gamma_n}\rangle = \langle H_1, \ldots,  H_n\rangle$.
In particular if these components are all equal, then the unique component is defined over $k$.    
\end{theorem}

\begin{remark}[Alternative patching characterization]\label{rem:equiv-char} The components $x_1^{\gamma_1}\cdots x_n^{\gamma_n}$ can  also be characterized by the following property (not using  Wewers' compactification):  

\vskip 1mm

\noindent
(*) \emph{they have points corresponding to G-covers of $\Pp^1$ obtained by patching a cover in $x_1$, a cover in $x_2$, ..., a cover in $x_n$, over some complete valued field containing $k$}.\footnote{We use here the patching procedure for $n$ covers and not just the one for $n=2$ presented in \S \ref{sec:patching}.}
\vskip 1mm

\noindent
The fact that this property implies the characterization from Theorem \ref{thm:cau}(a) above follows from considerations in rigid analytic geometry, as explained in \cite[\S 5.2-Step 4]{fielddef}. The converse follows from deformation techniques for covers over complete valued fields, as described and used for example in \cite{MR1841345}. 
Note further that if (*) characterizes the components $x_1^{\gamma_1}\cdots x_n^{\gamma_n}$, the following consequently holds:
\vskip 1mm

\noindent
(**) \emph{{\underbar{all}} points corresponding to G-covers of $\Pp^1$ obtained by patching a cover in $x_1$, a cover in $x_2$, ..., a cover in $x_n$, over \emph{\emph{\underbar{any}}} complete valued field containing $k$, lie in some component $x_1^{\gamma_1}\cdots x_n^{\gamma_n}$ for some $\gamma_1,\ldots,\gamma_n \in G$.}
\end{remark}

\subsection{Sketch of proof of main result} \label{sec:proof-main-theorem} For more details, see \cite{fielddef}.

\subsubsection{Part {\rm (a)} of Theorem \ref{thm:seguin2023}} \label{ssec:proof-part(a)} Braid calculations show that under the assumption that $\langle H, K \rangle = HK$, the set of components 
$x^{\gamma} y^{\delta}$ with  
$\gamma, \delta\in \langle H, K\rangle$ such that 
$\langle H^\gamma, K^\delta \rangle = \langle H, K\rangle$ is the singleton $\{xy\}$; see \cite[\S 3.2]{fielddef}. The conclusion then follows from Theorem  \ref{thm:cau}(b).

\subsubsection{Part {\rm (b)} of Theorem \ref{thm:seguin2023}} \label{ssec:proof-part(b)} As in part (a), consider the set of components $x^{\gamma} y^{\delta}$ with  $\gamma, \delta$ in $\langle H, K\rangle$ and such that 
$\langle H^\gamma, K^\delta \rangle = \langle H, K\rangle$. Clearly, the following data attached to $x^{\gamma} y^{\delta}$ do not depend on $\gamma, \delta$:

\noindent
- the group $\langle H, K\rangle$ of the G-covers parametrized by $x^{\gamma} y^{\delta}$; denote it by $\Gamma$,

\noindent
- the inertia canonical invariant ${\rm ICI}(x^{\gamma} y^{\delta})$; denote it by $\underline C$.

We now explain that it is also true of

\noindent
- the \emph{lifting invariant} of $x^{\gamma} y^{\delta}$.

\noindent
Denote by $c$ the union of conjugacy classes of $\Gamma$ appearing in $\underline C$, and by $M$ the minimal number of times that a given conjugacy class from $\underline C$ appears in $\underline C$. A central ingredient for this step and for the end of the argument is the \textit{{Conway-Parker-Fried-V\"olklein theorem}}, which we explain below, as presented in \cite{ellenberg-et-al-II} and \cite{wood2021}.

Let $U(\Gamma, c)$ be the group defined by this presentation: generators are all symbols $\hat g$ with $g\in c$, subject to the relations $\hat g \hat h \hat g^{-1} = \widehat{ghg^{-1}}$ for all $g, h \in c$. The \emph{$(\Gamma,c)$-lifting invariant} of a tuple $\underline g = (g_1, . . . , g_n) \in c^n$ is the element $\hat g_1\cdots \hat g_n\in U(\Gamma,c)$.
The $(\Gamma,c)$-lifting invariant of $\underline g$ is easily checked to only depend on the braid orbit of $\underline g$ (\hbox{e.g.} \cite[Prop.4.2]{fielddef}). 

Let ${\rm Comp}(\Gamma,c) \subset {\rm Comp}_G$ be the submonoid of components $[\underline g] = [(g_1,\ldots,g_n)]$ such that $n\in \Nn$, the group $\langle g_1,\ldots,g_n\rangle$ is contained in $\Gamma$ and $g_1,\ldots,g_n\in c$.  From above, the lifting invariant of a component $z = [(g_1, . . . , g_n)] \in {\rm Comp}(\Gamma,c)$ is well-defined. Further braid-like calculations based on the definition of $U(\Gamma, c)$ show that the $(\Gamma,c)$-lifting invariant
\vskip 1mm

\centerline{${\rm Comp}(\Gamma,c) \rightarrow U(\Gamma, c)$}
\vskip 1mm 

\noindent
 is multiplicative, and invariant by conjugation by any element of $\Gamma$ (\hbox{e.g.} \cite[Prop.4.3]{fielddef}). 
 
 Thus we indeed obtain that 
 all our components $x^\gamma y^\delta$ have the same $(\Gamma,c)$-lifting invariant. The Conway-Parker-Fried-V\"olklein theorem, stated below as Theorem \ref{thm:CPFV}, then yields that all components $x^\gamma y^\delta$ are equal to $xy$. The proof of Theorem \ref{thm:seguin2023}(b) can then be concluded  as in part (a)  by using Theorem \ref{thm:cau}(b): the component $xy$ is defined over $k$.
 
 \begin{theorem} \label{thm:CPFV} There exists an integer $M_{\Gamma,c}$ such that if $M\geq M_{\Gamma,c}$, the $(\Gamma,c)$-lifting invariant ${\rm Comp}(\Gamma,c) \rightarrow U(\Gamma, c)$ is injective when restricted to the subset of ${\rm Comp}(\Gamma,c)$ of components with group $\Gamma$ and a given inertia canonical invariant $\underline C$. 
 \end{theorem}

\begin{remark}[Versions of the Conway-Parker-Fried-V\"olklein theorem] \label{rem:CPFV} More refined versions than Theorem \ref{thm:CPFV} identify the image of the map ${\rm Comp}(\Gamma,c) \rightarrow U(\Gamma, c)$ as a specific quotient $H_2(\Gamma,c)$ of the second homology group $H_2(\Gamma,\Zz)$ (also called the Schur multiplier of $\Gamma$); see \cite[Theorem 7.6.1]{ellenberg-et-al-II}, \cite[Theorem 3.1]{wood2021}, \cite[Theorem 4.5]{fielddef}. In particular, this consequence immediately follows:
\vskip 1mm

\noindent
(*) \emph{Assume that $H_2(\Gamma,\Zz)$ is trivial, that all nontrivial conjugacy classes of $\Gamma$ appear at least 
$M_{\Gamma,c}$ times and that $\underline C$ is globally $k$-rational \hbox{\rm (\S \ref{ssec:ici})}. Then the Hurwitz space ${\rm Hur}_{r,\Gamma}(\underline C)$ is irreducible and defined over $k$.} 
\vskip 1mm

\noindent
Indeed there can be only one component in ${\rm Comp}(\Gamma,c)$ with ICI equal to $\underline C$.
\vskip 1mm

Statement (*) is the version of the Conway-Parker-Fried-V\"olklein theorem alluded to in the proof of Theorem \ref{thm:thm-Qtr-Qtp}. One may indeed assume there that $H_2(\Gamma,\Zz)$ is trivial: every finite group $G$ is known to be a quotient of a group $\Gamma$ satisfying this condition.
\end{remark}

\subsubsection{Part {\rm (c)} of Theorem \ref{thm:seguin2023}} \label{ssec:proof-part(c)}
The proof starts with a construction using Hilbert's irreducibility theorem. Let $r$ and $s$ be the branch point numbers of the G-covers parametrized by $x$ and $y$ respectively. From Hilbert's irreducibility theorem, applied  to the $k$-G-covers $x \rightarrow {\rm Conf}_r$ and $y \rightarrow {\rm Conf}_s$, one can find many unordered tuples in ${\rm Conf}_r(k)$ and ${\rm Conf}_s(k)$ such that the respective fibers in $x$ and $y$ consist of full orbits of ${\rm Gal}(\overline k/k)$. One can further require that the fields of definition of points in these fibers be linearly disjoint from any number field. Using this conclusion, one can construct two sequences of G-covers $(f_n)_{n\geq 1}$ and $(g_n)_{n\geq 1}$ and a sequence of number fields $(k_n)_{n\geq 1}$ such that for every $n\geq 1$,

\noindent
- $f_n$ and $g_n$ are $k_n$-G-covers in $x$ and $y$ respectively,

\noindent
- $k_n$ is linearly disjoint with the Galois closure of $k_1\cdots k_{n-1}$ over $k$.

\vskip 1mm
Patch the $k_n$-G-covers $f_n$ and $g_n$ over the complete valued field $k_n((u))$ ($n\geq 1$). The resul\-ting cover is a G-cover defined over $k_n((u))$, of group $\langle H,K\rangle$. Furthermore, from Theorem \ref{thm:cau} and Remark \ref{rem:equiv-char}, the resulting cover lies in some component  $x^{\alpha_n} y^{\beta_n}$ for some $\alpha_n,\beta_n\in G$. Then for two distinct integers $n, m\geq 1$, we have $(\alpha_n,\beta_n)=(\alpha_m,\beta_m) = (\alpha,\beta)$. By construction, the component $x^{\alpha} y^{\beta}$ has rational points over the fields $k_n((u))$, $k_m((u))$ and $\overline k$. Thus the component $x^{\alpha} y^{\beta}$ is defined over $\overline k \cap 
k_n((u)) \cap k_m((u))=k$. So is $x \hskip 1pt y^{\beta \alpha^{-1}}$.

\section{Recent advances II: Rational points over finite fields} \label{sec:finite-fields}

Investigating the growth of the number of components of Hurwitz spaces
when the branch point number $r$ tends to $\infty$ has been another fruitful development.

The Ellenberg-Venkatesh-Westerland result on the stability of homology of Hurwitz spaces was 
a breakthrough, notably towards the study of rational points on Hurwitz spaces over finite fields. 
The application to the Cohen-Lenstra heuristics for function fields over finite fields is quite striking.
Section \ref{ssec:EVW-stability} is devoted to their work. 

Their approach has then led to 
significant progress on the Malle conjecture  for function fields over finite fields. Section \ref{ssec: Malle} 
is devoted to this. Section \ref{ssec:counting components} concludes this paper with a general count of components when the basic Ellenberg-Venkatesh-Westerland \emph{non-splitting assumption} 
guaranteeing  boundedness of the number of components is removed, and with some related comments.

\subsection{Ellenberg-Venkatesh-Westerland approach} \label{ssec:EVW-stability}

\subsubsection{The homology stability result}
Fix a finite group $G$ and a conjugacy class $c$ of $G$ that generates $G$. 
Assume that this {{\it non-splitting property}} holds: 
\vspace{1mm}

\noindent
(Non-Split) {\it For every subgroup $H \subset G$, if $c\cap H\not= \emptyset$, then $c\cap H$ is a conjugacy class of $H$.} 
\vskip 1mm

\noindent
That is the case for example if $G$ is of order $2s$ with $s$ odd and $c$ is the involution class.

Consider the Hurwitz spaces ${\rm Hur}_{r,G}(c^r)$ of G-covers  with $r\geq 2$ branch points and branch cycles in $c$. We wish to investigate the number of components in ${\rm Hur}_{r,G}(c^r)$ when $r$ grows. A significant result from \cite{EVW2016} is this stabilization result.

\begin{theorem} \label{thm:EVW-main}
There exists a positive integer $u$ such that the number of com\-po\-nents of Hurwitz spaces ${\rm Hur}_{r,G}(c^r)$ becomes $u$-periodic as $r$ gets suitably large. The same is true with dimensions of higher homology groups of ${\rm Hur}_{r,G}(c^r)$.
\end{theorem}

\subsubsection{Stability of number of components} The whole proof of Theorem \ref{thm:EVW-main} is out of the scope of this paper, but the following starting arguments may give some flavor of it. The proof rests on a pure braid lemma \cite[Proposition 3.4]{EVW2016}.

\begin{lemma}\label{lem:basicEVW}
Let $g\in c$. For every suitably large integer $r$, every $r$-tuple $(g_1,\ldots,g_r)\in c^r$ such that $\langle g_1,\ldots,g_r \rangle = G$ is equivalent under the braid action of $B_r$ to a $r$-tuple $(g,g^\prime_2,\ldots, g^\prime_r)$ such that $\langle g^\prime_2,\ldots, g^\prime_r\rangle = G$.
\end{lemma}

Denote by ${\rm Comp}^C_{r,G}(c^r)$ the set of components of the space ${\rm Hur}_{r,G}(c^r)$ of \emph{connected to\-po\-logical G-covers of the affine line $\Aa^1(\Cc)$}.\footnote{\cite{EVW2016} starts with covers of $\Aa^1(\Cc)$.  In this subsection (only), we slightly abuse notation to use the same notation for this affine situation (which in addition we have not formally introduced).}  Lemma \ref{lem:basicEVW} shows that, for $g\in c$, the application 
\vskip 1mm

\centerline{${\rm Comp}^C_{r,G}(c^r) \rightarrow {\rm Comp}^C_{r+1,G}(c^{r+1})$}
\vskip 1mm

\noindent
mapping a component $x$ of degree $r$ (\hbox{i.e.} with $r$ branch points) to the degree $r+1$ component $ [g] x$ is surjective. As the number of components for each fixed $r$ is finite, this application becomes a bijection for suitably big $r$: \emph{the zeroth Betti number of the subset of  ${\rm Hur}_{r,G}(c^r)$
of \emph{connected} covers stabilizes as $r$ grows.}

\begin{proof}[Proof of Lemma \ref{lem:basicEVW}] The proof goes via the following simple braid argument. We use the symbol ``$\sim$'' for braid equivalence.

Assume that $r > d \hskip 2pt {\rm card}(c)$ where $d$ is the order of elements of $c$.
Then at least one of $g_1,\ldots,g_r$, say $g^\prime$, appears at least $d+1$ times 
in $\underline g= (g_1,\ldots,g_r)$. Then we have:

\vskip 0,5mm

\hskip 49mm ${g^\prime}$

\vskip -1,5mm

\hskip 49mm {\tiny $||$}
\vskip -1mm

\noindent
 \hskip 29mm $\underline g \sim (g^\prime,\ldots,g^\prime, h_1,\ldots,h_k)$ \hskip 18mm {\small (with $g^\prime$ repeated $d+1$ times)}
\vskip 1mm

\hskip 27mm $\sim (h_1,\ldots,h_k, g^\prime,\ldots,g^\prime)$ \hskip 18mm {\small (as $(g^\prime)^d = 1$)}
\vskip 1mm

\hskip 27mm $\sim (h g^\prime h^{-1},\ldots,h g^\prime h^{-1}, h_1,\ldots,h_k)$ \hskip 2mm {\small with $h=h_1\cdots h_k$.}

\vskip 1,5mm

\noindent
Note next that the same is true with $h=h_1\cdots h_{k-1}$ instead of $h=h_1\cdots h_k$, so also with $h=h_k$, with $h=h_{k-1}$ too, etc. Conclude that $\underline g \sim (\gamma,\ldots,\gamma,h_1, \ldots, h_k)$ for $\gamma$ any conjugate of $g^\prime$ by an arbitrary element of $\langle h_1,\ldots,h_k\rangle = G$, in particular for $\gamma = g$.
\end{proof}

Set ${\rm Comp}_{G}^C(c) = \bigcup_{r\geq 1} {\rm Comp}_{r,G}^C(c^r)$. A clever but elementary argument next uses Lemma \ref{lem:basicEVW} to deduce that there is a central homogeneous element $U$ in the sub-graded ring $k[{\rm Comp}_{G}^C(c)]$ of $k[{\rm Comp}_G]$, of positive degree and such that   multiplication by $U$
\vskip 1mm

\centerline{$\rho \in k[{\rm Comp}_{G}(c)] \hskip 2mm \mapsto \hskip 2mm  U \rho \in k[{\rm Comp}_{G}(c)]$}
\vskip 1mm

\noindent
has finite $k$-dimensional kernel and cokernel. This part of the proof uses in an essential way the non-splitting assumption; see \cite[Proposition 3.4]{EVW2016}. This element $U\in k[{\rm Comp}_{G}(c)]$ is an important tool for the rest of the proof of Theorem  \ref{thm:EVW-main}; in particular, the number of com\-po\-nents of Hurwitz spaces ${\rm Hur}_{r,G}(c^r)$ is $u$-periodic for $u=\deg(U)$.

\subsubsection{Points over finite fields} 

Let $\Ff_q$  be a finite field of characteristic $p$. Denote by $X_r$ the (possibly reducible) $\Ff_q$-variety
$X_r= \mathcal{H}_p(c^r)\otimes_{\Ff_p} \Ff_q$ where $\mathcal{H}_p(c^r) = \mathcal{H}ur_{r,G}(c^r) \otimes_{\Zz[1/|G|]} {\Ff_p}$ 
(as in Theorem \ref{thm:wewers-romagny}). 
Fix an odd prime $\ell\not=p$. The  Grothendieck-Lefschetz fixed point formula gives
\vskip 0,5mm 

\centerline{{$\displaystyle \# X_r(\Ff_q) = \sum_{i=0}^{2r} (-1)^i \hskip 1mm \hbox{Tr}(\hbox{Frob},H^i_c(X_r \otimes_{\Ff_q}\overline{\Ff_q}, \Qq_\ell))$,}} 
\vskip 1mm

\noindent
where  $\hbox{Tr}(\hbox{Frob},H^i_c(X_r \otimes_{\Ff_q}\overline{\Ff_q}, \Qq_\ell))$ is the trace of the Frobenius acting on the compactly supported \'etale $i$-th cohomology group.
The main term in the sum is for {$i =2r$}: it is 
\vskip 1mm

\centerline{$q^{r}$ $\times$ $\#\{\hbox{components of}\hskip 2pt X_r\hskip 2pt  \hbox{defined over } \Ff_q\}$.}
\vskip 1mm

\noindent
For suitably big $r$, this term is here $q^r$; as shown in the proof \cite[Theorem 8.8]{EVW2016} (see pp.42-43), $X_r$ has only one component defined over $\Ff_q$.\footnote{The issue of counting components defined over $\Ff_q$ is more generally investigated by Liu-Wood-Zureick-Brown \cite{LiuWoodZureickBrown}; see in particular Theorem 10.4 there.}
Use then the Deligne bound for the eigenvalues of the Frobenii: they are of module
\vskip 1mm

\centerline{$\leq \sqrt{q^{2r-i}}$ \hskip 3mm {\footnotesize (for $(2r-i)$-th term).}}
\vskip 1mm

\noindent
The next step consists in bounding the $\Qq_\ell$-dimensions of the cohomology groups. Using the Poincar\'e duality, it follows from the stability of singular homology groups (under the \hbox{non-splitting} hypothesis) that we have an upper bound of the form
\vskip 1mm

\centerline{$\leq C^i$  \hskip 4mm{\footnotesize (for $(2r-i)$-th term) for some $C=C(G,c)$;} }
\vskip 1mm

\noindent
This leads to the following result, stated in the proof of \cite[Theorem 8.8]{EVW2016}.

\begin{theorem} Let $A$ be a nontrivial finite abelian $\ell$-group,   
$G= A\rtimes \Zz/2\Zz$,  and $c$ the conjugacy class of $G$ of involutions. For some constant $C$ and all suitably large $r$ and $q$ (relative to $A$), we have:
$$\left| \frac{\# {\mathcal{H}ur}_{r,G}(c^r)({\Ff_q})}{q^r} - 1 \right| \leq  \frac{C}{\sqrt{q}}$$
\end{theorem}
\vspace{0mm}

\subsubsection{From Hurwitz Spaces to Cohen-Lenstra} The paper \cite{EVW2016} has a main arithmetic motivation: to prove,  in the context of function fields over  finite fields, the Cohen-Lenstra heuristics on the distribution of class groups of quadratic extensions among all finite abelian $\ell$-groups (originally conjectured for number fields). The specific result they obtain is the following asymptotic result (which is  \cite[Theorem 1.2]{EVW2016}) about quadratic extensions $L/\Ff_q(t)$ given with a surjection ${\rm Cl}_L \rightarrow A$ of the class group of $L$ onto $A$.

\begin{theorem}\label{thm:EVW2016} Let $\ell > 2$ be prime and $A$ be a finite abelian $\ell$-group. 
Write $\delta^+$ (resp. $\delta^-$) for the upper  density (resp. lower density) of  quadratic extensions of $\Ff_q(t)$ ramified at $\infty$ and for which the $\ell$-part of the class group is isomorphic to $A$. Then $\delta^+(q)$ and $\delta^-(q)$ converge as $q \rightarrow \infty$ with $q\not=1 \pmod{\ell}$ to 

\vskip 1mm

\centerline{$\displaystyle \prod_{i\geq 1} (1-\ell^{-i}) / |{\rm Aut}(A)|$.}
\end{theorem}

\noindent
In particular, for all suitably large $q$ (relative to $A$), a positive fraction of quadratic extensions  
$L/\Ff_q(t)$ ramified at $\infty$ have the $\ell$-part of their class group isomorphic to $A$.

\begin{proof}[The Hurwitz space argument] We focus on the part of the proof involving Hurwitz spaces. 
Let ${\mathcal M}_A$ be the set of isomorphism classes of pairs  $(L, \alpha)$ (for appropriate isomorphisms), where $L/\Ff_q(t)$ is a quadratic extension of discriminant degree $r+1$ ramified at $\infty$, and $\alpha$ a surjective homomorphism $\alpha:{\rm Cl}_L \rightarrow A$. \cite[Prop.8.7]{EVW2016} 
produces a bijection 
\vskip 1mm
\centerline{${\mathcal{H}ur}_{r,G}(c^r)({\Ff_q}) \rightarrow {\mathcal M}_A$.}
\vskip 1mm

\noindent
Using a result of Carlitz, one obtains that the total number of quadratic extensions of discriminant degree $r+1$ ramified at $\infty$ is $q^r-q^{r-1}$. This leads to

\vskip 2mm

\hskip 25mm $\displaystyle \left|\frac{\# {\mathcal M}_A}{q^r-q^{r-1}} -1 \right|  \approx \left| \frac{\# {\mathcal{H}ur}_{r,G}(c^r)({\Ff_q})}{q^r} - 1 \right| \leq  \frac{C(A)}{\sqrt{q}}$.
 \end{proof}

\subsection{Malle conjecture over function fields} \label{ssec: Malle} \cite{EVW2016} opened the way to several celebrated problems in arithmetic statistics. One of them is the Malle conjecture in inverse Galois  theory. This conjecture is a quantitative version of the Inverse Galois Problem: it predicts an estimate for the number of extensions of a global field such that the Galois closure has a given Galois group $G$ and the discriminant is bounded; only finitely many such extensions exist by a result of Hermite.

Specifically, fix a transitive subgroup $G\subset S_m$ and use the letter $K$ for $\Qq$ or $\Ff_q(T)$. 
Set

\vskip 2mm

\centerline{$N(K,G,X)= \# \left\{ E/K\hskip 1mm  \middle\vert \hskip -3mm{\scriptsize \begin{matrix} & \hbox{ of degree } m  \hfill \cr & \hbox{of Galois group } G \cr & |N_{K/\Qq}(d_{E/K})| \leq X \cr \end{matrix}}
\right\}$ }
\vskip 2mm

\noindent
The Malle conjecture predicts that $N(K,G,X)$ can be bounded from below and from above by some function of the form $C(K,G) \hskip 1pt X^{a(G)} \log X^{b(K,G)}$ for some constants $C(K,G)$, $a(G)$, $b(K,G)$, the last two being explicitly given in the conjecture ; see \cite{Ma} and \cite{Ma2}. A \emph{strong form} of Malle's conjecture predicts that $N(K,G,X)$ is \emph{asymptotic} to such a function $C(K,G) \hskip 1pt  X^{a(G)} \log X^{b(K,G)}$. The leading constant $C(K,G)$ was recently investigated by Loughran-Santens who put forward a specific conjectural value \cite{loughran-santens}.

Over $K=\Qq$, the Malle conjecture is known in few cases, each of them being a significant result: if $G$ is abelian (M\"aki \cite{maki}, Wright \cite{Wright} in the strong form), if $G=S_3$ (Davenport-Heilbronn \cite{Davenport-Heilbronn}), if $G=S_4,S_5 $ (Bhargava \cite{Bhargava2005,Bhargava2010}), if $G$ is nilpotent \cite{KluenersMalle}. See also \cite{alberts-lemke-wang-wood} for more recent contributions.

For $K=\Ff_q(T)$, the Ellenberg-Venkatesh-Westerland method has rapidly risen hopes that the approach via  $\Ff_q$-rational points on Hurwitz spaces (which correspond to finite extensions of $\Ff_q(T)$) should lead to some significant progress. However, one had to overcome the very stringent non-splitting assumption that appears in the Ellenberg-Venkatesh-Westerland method. This is what Landesmann and Levy have achieved \cite{landesman2025homologicalstabilityhurwitzspaces}. As a result of their work, combined with further recent works, the strong form of the Malle conjecture over $\Ff_q(T)$ is now established, except for some restrictions on $q$.

\begin{theorem} For every finite group $G$, there is a constant $q_0(G)$ such that for all prime powers $q\geq q_0(G)$, prime to $|G|$, and for $K=\Ff_q(T)$, the number $N(K,G,X)$ is asymptotic to $C(K,G) \hskip 1pt  X^{a(G)} \log X^{b(K,G)}$, where $a(G)$, $b(G,K)$ are the expected values for the exponents of $X$ and $\log X$ and the leading constant $C(K,G)$ is also explicit.
\end{theorem}

More precisely, a version with only the upper bound part was first obtained by Ellenberg-Tran-Westerland \cite{ellenberg2023foxneuwirthfukscellsquantumshuffle}. Then Landesmann and Levy, thanks to their breakthrough result on the stabilization of homology of Hurwitz spaces (parametrizing connected covers), could prove the  ``bounded from below and from above'' version in \cite{landesman2025homologicalstabilityhurwitzspaces}. Finally Santens proved the strong form stated above in \cite{santens2026leadingconstantmallesconjecture}, for the $\Ff_q(T)$-analog of the leading constant $C(K,G)$ that he and Loughran had put forward in \cite{loughran-santens}.

\begin{remark}[variants] 
The prime-to-$p$ assumption on $|G|$ is due to the use of Hurwitz spaces,
which only parametrize tame covers. The asymptotic distribution of wildly ramified extensions of function fields in characteristic $p \geq 2$ has been investigated by Gundlach-Seguin \cite{wildcount}, for certain $p$-groups of nilpotency class at most $2$. A key ingredient replacing Hurwitz spaces is Abrashkin's nilpotent Artin-Schreier theory.
In an earlier paper \cite{asympskew}, Gundlach-Seguin also considered the distribution of extensions of a given central simple algebra $K$ over a number field, thus extending the asymptotic issues \`a la Malle to the noncommutative context.
\end{remark}

\subsection{Splitting number and component count} \label{ssec:counting components}

In \cite{countcomp}, Seguin con\-si\-ders the growth of the number of components of Hurwitz spaces (parametrizing not necessarily connected covers)  when the non-splitting assumption is removed. 

Fix a group $G$ and let $c$ be a finite set of conjugacy classes of $G$.

\begin{definition}
Given a subgroup $H\subset G$, the \emph{splitting number}  $\Omega_c(H)$ of a subgroup $H\subset G$ is the difference between the number of conjugacy classes of $H$ contained in some class in $c$ and the number of conjugacy classes of $G$ in $c$. 
\end{definition}

The non-splitting assumption for $c$ a single conjugacy class means that $\Omega_c(H)=0$ for every subgroup $H\subset G$. As Seguin shows, the splitting number relates to the stratification of the scheme of components described in \S \ref{ssec:scheme-components}: $\dim(\gamma(H))=\Omega_c(H)+1$ for every subgroup $H\subset G$ \cite[Theorem 1.2]{countcomp}. The pertinence of the splitting number  is even more striking in \cite[Theorem 1.4]{countcomp} which counts the components of the Hurwitz space when the branch point number grows:

\begin{theorem} \label{thm:beranger-scheme}
Let $(\xi_C)_{C\in c}$ be a family of integers $\xi_C\geq 0$ indexed by $c$. Let ${\rm HF}_H(r)$ be the number of components of the Hurwitz space of {connected} $H$-covers of $\Pp^1$ such that, for each $C \in c$, each conjugacy class of $H$ in $C\cap H$ appears $r \xi_C$ times in the Inertia Canonical Invariant. Then the function ${\rm HF}_H$ is a $O(r^{\Omega_c(H)})$ and is not a $o(r^{\Omega_c(H)})$.
\end{theorem}

When $H$ ranges over all possible subgroups of $G$, one obtains an estimate for the dimension of the $k$-subspace of  elements of given degree of the ring of components, i.e., for the {H}{\it ilbert} {F}{\it unction} of the graded ring.
The function ${\rm HF}_H$ is bounded if $\Omega_c(H)=0$, notably under the non-splitting assumption and obviously for $H=G$.

Following Seguin's paper \cite{countcomp}, the splitting number was also shown to naturally appear in higher homology  by Bianchi-Miller \cite{bianchi2024polynomialstabilityhomologyhurwitz}.

\bibliography{2025_NAKAMURA60}
\bibliographystyle{alpha}

\end{document}